\documentclass[3p,authoryear]{elsarticle}
\usepackage{lineno,hyperref}
\modulolinenumbers[5]
\journal{Journal of Computational and Applied Mathematics}
\usepackage{amsfonts,amssymb,amsmath,graphicx,amsthm,float,graphicx,caption,subcaption}
\newtheorem{theorem}{Theorem}

\newtheorem{lemma}[theorem]{Lemma}
\newtheorem{definition}[theorem]{Definition}
\newtheorem{example}[theorem]{Example}
\newtheorem{remark}[theorem]{Remark}
\usepackage{algorithm,caption}
\usepackage[noend]{algpseudocode}

\hypersetup{
    colorlinks=true,
    linkcolor=blue,
    filecolor=magenta,      
    urlcolor=cyan,
    pdftitle={Overleaf Example},
    pdfpagemode=FullScreen,
    }
 \usepackage{bm}

\def\bfp{{\boldsymbol{p}}}
\def\bfq{{\boldsymbol{q}}}

\def\bfx{{\boldsymbol{x}}}

\def\bfu{{\boldsymbol{u}}}
\def\bfv{{\boldsymbol{v}}}
\def\bfw{{\boldsymbol{w}}}

\newcommand\scalemath[2]{\scalebox{#1}{\mbox{\ensuremath{\displaystyle #2}}}}
\usepackage{ifthen}
\begin{document}

\newcounter{num}
    \setcounter{num}{0}

\begin{frontmatter}

\title{{A new method} to detect projective equivalences and symmetries of rational $3D$ curves}


\author[FirstAdress]{U\u{g}ur G\"oz\"utok}\corref{mycorrespondingauthor}
\cortext[mycorrespondingauthor]{Corresponding author}
\ead{ugurgozutok@ktu.edu.tr}

\author[FirstAdress]{H\"usn\"u An{\i}l \c{C}oban}
\ead{hacoban@ktu.edu.tr}

\author[FirstAdress]{Yasemin Sa\u{g}{\i}ro\u{g}lu}
\ead{ysagiroglu@ktu.edu.tr}

\address[FirstAdress]{Department of Mathematics, Karadeniz Technical University, Turkey}

\author[a]{Juan Gerardo Alc\'azar}
\ead{juange.alcazar@uah.es}
\address[a]{Departamento de F\'{\i}sica y Matem\'aticas, Universidad de Alcal\'a,
E-28871 Madrid, Spain}

\begin{abstract}
We present a new approach to detect projective equivalences and symmetries between two rational parametric $3D$ curves properly parametrized. In order to do this, we introduce two {rational functions that behave nicely for} M\"obius transformations, which are the transformations in the parameter space  associated with the projective equivalences between the curves. The M\"obius transformations are found by first computing the gcd of two polynomials built from {these two functions}, and then searching for {a special type of factors, ``M\"obius-like", of this gcd}. The projective equivalences themselves are easily computed from the M\"obius transformations. In particular, and unlike previous approaches, we avoid solving big polynomial systems. The algorithm has been implemented in \citet{maple}, and evidences of its efficiency as well as a comparison with previous approaches are given. 
\end{abstract}


\end{frontmatter}


\section{Introduction.}\label{intro}

Detecting projective and affine equivalences implies recognizing whether or not two objects are the same in a certain setup. Also, finding the symmetries of an object is important in order to understand its shape, and also to efficiently visualize and store the information regarding the object. For these reasons, these questions have been treated in fields like Computer Vision, Computer Graphics, Computer Aided Geometric Design and Pattern Recognition. Several studies addressing the problem are, for instance, \citep{Bokeloh2009,Brass20043,Huang19961473,Lebmeir2008707,Lebmeir2009}; a more comprehensive review can be found in \citep{Alcazar201551}.

In recent years several papers \citep{Alcazar2014715,Alcazar2014b199,Alcazar2014a269,Alcazar201551,Alcazar2019775,ALCAZAR2019302,Hauer2019424,BIZZARRI2020101794,Hauer201868,DBLP,JUTTLER2022571} have pursued these problems for rational curves and surfaces, using tools from Algebraic Geometry and Computer Algebra. In the case of curves, the main idea behind these approaches is the fact that projective or affine equivalences between the curves, and symmetries as a particular case, have a corresponding transformation in the parameter domain which must be a M\"obius transformation whenever the curves are properly, i.e. birationally, parametrized. Thus, the usual approach is to compute the M\"obius transformations, and derive the equivalences themselves from there. 

For projective equivalences, the algorithms in \citep{Hauer201868,BIZZARRI2020112438} follow this strategy and compute the M\"obius transformations by solving a polynomial system which is increasingly big as the degree of the curves involved in the computation grows. Solving {this} polynomial system implies using Gr\"obner bases, which results in higher complexity. In this paper we use a different approach following the idea in \citep{Alcazar201551}, where the classical curvature and torsion, two well-known differential invariants, are used to compute the symmetries of a space rational curve. In \citep{Alcazar201551} the M\"obius transformations are derived as special factors of a gcd of two polynomials, computed from the curvature and torsion functions. On one hand, this has the advantage of working with smaller polynomials, since taking the gcd already reduces the degree of the polynomial one has to analyze. On the other hand, one avoids solving polynomial systems by using factoring instead. 

In a similar way, in this paper we present a strategy for 3D space rational curves that also pursues the M\"obius transformations first. However, in order to compute them we introduce two differential invariants, which we call {\it projective curvatures}, that allow us to obtain the M\"obius transformations using an analogous procedure to that in \citep{Alcazar201551}, i.e. using gcd computing and factoring over the reals, without sorting to polynomial system solving. The projective curvatures are {inspired by} ideas from differential invariant theory \citep{MR836734,dolgachev_2003,olver_1995,mansfield_2010}. To this aim, we first present four {rational expressions} that completely characterize projective equivalence but that, however, are not well suited for computation, because they do not 
{have a good behavior with respect to M\"obius transformations; in the terminology of this paper, we express this by saying that they ``do not commute" with M\"obius transformations.} From here, we develope two more {rational expressions}, the {\it projective curvatures}, that do commute with M\"obius transformations, and we characterize projective equivalence between the curves using these curvatures. {In particular, it is this good behavior with respect to M\"obius transformations that allows us to solve the problem by using gcd and factoring.} The experimentation carried out in \citet{maple} shows that our approach is efficient and works better than \citep{Hauer201868,BIZZARRI2020112438} as the degree of the curves grow. 

The structure of the paper is the following. In Section \ref{prelmn} we provide {some background on the problem treated in the paper, as well as some preliminary notions and results to be use later.} The main results behind the algorithm are developed in Section \ref{newmeth}, where we introduce several {rational expressions} to finally derive the projective curvatures, and the theorems relating them to the projective equivalences between the curves. The algorithm itself is provided in Section \ref{sec5}. We present the results of the experimentation carried out in \citet{maple} in Section \ref{Imp}, where a comparison with the results in \citep{Hauer201868,BIZZARRI2020112438} is also given. Finally, we close with our conclusion in Section \ref{sec-conclusion}. Several technical results and technical proofs are deferred to two appendixes, so as to improve the reading of the paper.

\section*{Acknowledgements}

The first three authors are partially supported by the project FAY-2021-9648 of Scientific Research Projects Unit, Karadeniz Technical University. The first author would like to thank TUBITAK (The Scientific and Technological Research Council of Turkey) for their financial supports during his doctorate studies. Juan G. Alc\'azar is supported by the grant PID2020-113192GB-I00 (Mathematical Visualization: Foundations, Algorithms and Applications) from the Spanish MICINN, and is also a member of the Research Group {\sc asynacs} (Ref. {\sc ccee2011/r34}). Juan G. Alc\'azar and U\u{g}ur G\"oz\"utok were also supported by TUBITAK for a short research visit to Yildiz Technical University in Istanbul (Turkey). {The authors are also thankful to the reviewers for their comments, which made it possible to improve the paper with respect to an earlier version.}

\section{Preliminaries.}\label{prelmn}


\subsection{General notions and assumptions.}\label{subsec1}

For the sake of comparability, in general we will follow the notation in \citep{Hauer201868}. Thus, let $\bm{C_1}$ and $\bm{C_2}$ 
be two parametric rational curves embedded in the three real projective space ${{\Bbb P}^3(\mathbb{R})}$. The points $\bm{x}\in{{\Bbb P}^3(\mathbb{R})}$ are represented by $\bm{x}=(x_0,x_1,x_2,x_3)^T$, where the $x_i$ are real numbers and correspond to the \emph{homogeneous} coordinates of $\bm{x}$. In particular, whenever $\lambda\neq 0$, the vectors $\bm{x}$ and $\lambda \bm{x}$ represent the same point in ${{\Bbb P}^3(\mathbb{R})}$. The curves $\bm{C_1}$ and $\bm{C_2}$ are defined by means of parametrizations

\begin{align*}
\bm{p}:{\Bbb P}^1(\mathbb{R})\rightarrow \mathbf{C_1} \subset{{\Bbb P}^3(\mathbb{R})}, &&(t_0,t_1)\rightarrow \bm{p}(t_0,t_1)=(p_0(t_0,t_1),p_1(t_0,t_1),p_2(t_0,t_1),p_3(t_0,t_1)), \\
\bm{q}:{\Bbb P}^1(\mathbb{R})\rightarrow \mathbf{C_2} \subset{{\Bbb P}^3(\mathbb{R})}, &&(t_0,t_1)\rightarrow \bm{q}(t_0,t_1)=(q_0(t_0,t_1),q_1(t_0,t_1),q_2(t_0,t_1),q_3(t_0,t_1)),
\end{align*}
where ${\Bbb P}^1(\mathbb{R})$ denotes the real projective line. {The components of each parametric map} are homogeneous polynomials of degree $n$,

\begin{align*}
p_i(t_0,t_1)=\sum_{j=0}^{n} c_{j,i}t_0^{n-j}t_1^{j} \textit{  and  } q_i(t_0,t_1)=\sum_{j=0}^{n} c'_{j,i}t_0^{n-j}t_1^{j},
\end{align*} 
with $i\in \{0,1,2,3\}$, and $c_{j,i},c'_{j,i}\in {\Bbb R}$. Additionally, we denote 
\begin{align*}
\bm{c_j}=(c_{j,0},c_{j,1},c_{j,2},c_{j,3})^T, \; 
\bm{c'_{j}}=(c'_{j,0},c'_{j,1},c'_{j,2},c'_{j,3})^T,
\end{align*}
which will be referred to as the \emph{coefficient vectors} of the curves. {\color{black} For instance, consider the parametrization 
\begin{equation*}
\bm{p}(t_0,t_1)=\begin{pmatrix}
t_0^4+t_1^4 \\
4t_0^3t_1 \\
-8t_0^2t_1^2 \\
t_0t_1^3-2t_1^4
\end{pmatrix}.
\end{equation*}
The coefficient vectors are $\bm{c_0}=(1,0,0,0)^T$, $\bm{c_1}=(0,4,0,0)^T$, $\bm{c_2}=(0,0,-8,0)^T$, $\bm{c_3}=(0,0,0,1)^T$, $\bm{c_4}=(1,0,0,-2)^T$. Hence, we can rewrite $\bm{p}$ as
\begin{equation*}
\bm{p}(t_0,t_1)=\begin{pmatrix}
1 & 0 & 0 & 0 & 1 \\
0  & 4 & 0 & 0 & 0 \\
0  & 0 & -8 & 0 & 0 \\
0 & 0& 0 & 1& -2
\end{pmatrix}\begin{pmatrix}
t_0^4 \\
t_0^3t_1\\
t_0^2t_1^2\\
t_0t_1^3\\
t_1^4
\end{pmatrix}.
\end{equation*}
}

Furthermore, we make the following assumptions on $\bm{C_1}$ and $\bm{C_2}$; we will refer later to these hypotheses as \emph{hypotheses (i-iv)}.
\begin{itemize}
\item [(i)] The parametrizations $\bm{p}$ and $\bm{q}$ defining $\bm{C_1}$ and $\bm{C_2}$ are \emph{proper}, i.e. birational, so that almost all points in  $\bm{C_i}$ {are reached by the parametrizations}. It is well-known that every rational curve can be reparametrized to obtain a proper parametrization \cite[see][]{Sendra2008}.
\item [(ii)] The parametrizations $\bm{p}$ and $\bm{q}$ are in reduced form, i.e.,
\begin{equation*}
gcd(p_0(t_0,t_1),p_1(t_0,t_1),p_2(t_0,t_1),p_3(t_0,t_1))
=gcd(q_0(t_0,t_1),q_1(t_0,t_1),q_2(t_0,t_1),q_3(t_0,t_1))=1.
\end{equation*} 
\item [(iii)] Both parametrizations $\bm{p}$ and $\bm{q}$  have the same degree $n$. Notice that since projective transformations preserve the degree, the degree of projectively equivalent curves must be equal. Furthermore, we assume $n\geq 4$. 
\item [(iv)] None of the $\bm{C_i}$ is contained in a hyperplane. Consequently, the matrices $(c_{j,k})$, $(c'_{j,k})$ formed by the coefficient vectors $\bm{c_j}$ and $\bm{c'_{j}}$ have rank $4$ \citep{Hauer201868}.  
\end{itemize}

\begin{remark}\label{isnotzero}
Notice that because of these assumptions, the coefficient vectors $\bm{c_0}$ and $\bm{c'_{0}}$ cannot be identically zero. 
\end{remark}

A \emph{projectivity} is a mapping $f$ defined in {${\Bbb P}^3(\mathbb{R})$} such that 
\begin{align*}
f:{{\Bbb P}^3(\mathbb{R})}\rightarrow {{\Bbb P}^3(\mathbb{R})}:\bm{x}\mapsto f(\bm{x})=M\cdot\bm{x},
\end{align*}
where $M=(m_{ij})_{0\leq i,j\leq 3}$ is a non-singular $4\times 4$ matrix. If $m_{00}\neq 0,\,\, m_{01}=m_{02}=m_{03}=0$, then $f$ is an affine transformation. {Observe that the matrix $M$ is defined up to a nonzero scalar, so that $M,\mu M$ with $\mu\in {\Bbb R}-\{0\}$ define the same projectivity.} Then we have the following definition. 

\begin{definition}\label{def3}
Two curves $\bm{C_1}$ and $\bm{C_2}$ are said to be \emph{projectively equivalent} if there exists a projectivity $f$ such that $f(\bm{C_1})=\bm{C_2}$. A curve $\bm{C}$ has a \emph{projective symmetry} if there exists a non-trivial projectivity $f$ such that $f(\bm{C})=\bm{C}$.
\end{definition}

It is well-known \citep{Sendra2008,Hauer201868,BIZZARRI2020112438} that any two proper parametrizations of a rational curve are related by a linear rational transformation
\begin{equation}\label{eq-moeb}
\varphi:{\Bbb P}^1(\mathbb{R})\to {\Bbb P}^1(\mathbb{R}),\quad (t_0,t_1)\to\varphi(t_0,t_1)=(at_0+bt_1,ct_0+dt_1),
\end{equation}
with $ad-bc\neq 0$. {Notice that $a,b,c,d$ are defined up to a common nonzero scalar factor.} The mapping $\varphi$ is called a \emph{M\"obius transformation}. This fact is essential to prove the following result, which is used in \citep{Hauer201868,BIZZARRI2020112438}.

\begin{theorem}\label{theo0.1}
Two rational curves $\bm{C_1},\bm{C_2}$ properly parametrized by $\bm{p}$ and $\bm{q}$ are projectively equivalent if and only if there exist a non-singular $4\times 4$ matrix $M$ and a M\"obius transformation $\varphi(t_0,t_1)=(at_0+bt_1,ct_0+dt_1)$ with $ad-bc\neq 0$ such that 
\begin{equation}\label{eq0.1}
M\bm{p}={\bm{q}\circ\varphi}.
\end{equation} 
\end{theorem}

\begin{color}{black}
\begin{remark}\label{lambd}
In general, projective equivalence implies that $M\bm{p}=\lambda(\bm{q}\circ \varphi)$ with $\lambda\neq 0$. However, since we are working in a projective setting, we can always safely assume that $\lambda=1$; this is the case in Eq. \eqref{eq0.1}.
\end{remark}
\end{color}

\subsection{{A notion of invariance.}}\label{subsec2.2}

{\color{black}
We are interested in building rational expressions in terms of the components of the parametrizations that are invariant under projectivities, and that can help us recognize when two given rational curves are projectively equivalent. In more detail, let $\bm{u}=\bm{u}(t_0,t_1)$, $\bm{v}=\bm{v}(t_0,t_1)$ be two homogeneous parametrizations in ${\Bbb P}^3({\Bbb R})$, defining two curves $\bm{D_1},\bm{D_2}$, and assume that $M \bm{u}=\bm{v}$ with $\mbox{det}(M)\neq 0$. Therefore, according to Definition \ref{def3} the curves $\bm{D_1},\bm{D_2}$ are projectively equivalent. We want to build expressions $I_i$, which are rational in the components of $\bfu$ and its derivatives, so that $I_i(\bm{u})=I_i(M\bm{u})$ for all non-singular $4\times 4$ matrices $M$. This helps to recognize whether or not $\bm{u},\bm{v}$ satisfy that $M \bm{u}=\bm{v}$ for some $M$, since $I_i(\bm{u})=I_i(\bm{v})$ are necessary conditions for this. We say that $I_i$ is \emph{invariant} under projectivities, and we will often refer to $I_i$ simply as an \emph{invariant}. Since the $I_i$ are rational functions of $\bfu$ and its derivatives, when considering rational parametrizations $\bfu$, the $I_i$ are, in turn, rational functions. 

In order to build expressions of this type, we will use two ingredients. The first one has to do with determinants. Let $\bfw_i\in {\Bbb R}^n$ for $i=1,\ldots,n$, and let $\Vert\bfw_1\,\ldots\,\bfw_n\Vert$ denote the determinant of the vectors $\bfw_i$; we will keep this notation in the rest of the paper. Let $A\in {\mathcal M}_{n\times n}({\Bbb R})$ be a matrix whose determinant $\mbox{det}(A)$ is nonzero. Then we recall that
\begin{equation}\label{ref-det}
\Vert A\bfw_1\, \ldots \, A\bfw_n\Vert=\mbox{det}(A) \Vert\bfw_1\, \ldots\,\bfw_n\Vert.
\end{equation}

The second ingredient is differentiation. Let us denote, and we will also keep this notation in the rest of the paper,
\begin{equation}\label{notation}
\bm{u}_{t_0^k t_1^l}=\dfrac{\partial^{k+l} \bm{u}}{\partial t_0^k\partial t_1^l}(t_0,t_1),
\end{equation}
and analogously for $\bfv$. By repeatedly differentiating the equality $M \bm{u}=\bm{v}$ with respect to $t_0,t_1$, we get that 
\begin{equation}\label{ing2}
M \bm{u}_{t_0^k t_1^l}= \bm{v}_{t_0^k t_1^l}
\end{equation}
for any choice $(k,l)\in ({\Bbb Z}^+\cup \{0\})\times ({\Bbb Z}^+\cup \{0\})$. 

Now let us see how to use these two ingredients to build rational expressions with the desired property. We will show it with one example. Let us pick four different elements $(k,l)\in ({\Bbb Z}^+\cup \{0\})\times ({\Bbb Z}^+\cup \{0\})$, say $\{(4,0), (0,1),(2,0),(3,0)\}$, which according to Eq. \eqref{notation} correspond to the derivatives $\bm{u}_{t_0^4}
, \bm{u}_{t_1}, \bm{u}_{t_0^2}, \bm{u}_{t_0^3}$. Because of Eq. \eqref{ref-det}, we get that 
\begin{equation}\label{ba1}
\Vert M\bm{u}_{t_0^4}\, M \bm{u}_{t_1}\, M\bm{u}_{t_0^2}\, M\bm{u}_{t_0^3}\Vert=\mbox{det}(M) \Vert\bm{u}_{t_0^4}
\,\bm{u}_{t_1}\, \bm{u}_{t_0^2}\, \bm{u}_{t_0^3}\Vert.
\end{equation}
Taking Eq. \eqref{ing2} into account, 
\begin{equation}\label{ba2}
\Vert M\bm{u}_{t_0^4}\, M \bm{u}_{t_1}\, M\bm{u}_{t_0^2}\, M\bm{u}_{t_0^3}\Vert= \Vert\bm{v}_{t_0^4}
\, \bm{v}_{t_1}\, \bm{v}_{t_0^2}\, \bm{v}_{t_0^3}\Vert.
\end{equation}
Therefore, from Eq. \eqref{ba1} and Eq. \eqref{ba2}, we get that 
\begin{equation}\label{ba3}
\mbox{det}(M)\Vert\bm{u}_{t_0^4}
\, \bm{u}_{t_1}\, \bm{u}_{t_0^2}\, \bm{u}_{t_0^3}\Vert=\Vert\bm{v}_{t_0^4}
\, \bm{v}_{t_1}\, \bm{v}_{t_0^2}\, \bm{v}_{t_0^3}\Vert.
\end{equation}

Any other choice of four different elements $(k,l)\in ({\Bbb Z}^+\cup \{0\})\times ({\Bbb Z}^+\cup \{0\})$ provides a relationship similar to Eq. \eqref{ba3}. For instance, if we pick $\{(1,0),(0,1),(2,0),(3,0)\}$ then we get 
\begin{equation}\label{ba4}
\mbox{det}(M)\Vert\bm{u}_{t_0}
\, \bm{u}_{t_1}\, \bm{u}_{t_0^2}\, \bm{u}_{t_0^3}\Vert= \Vert\bm{v}_{t_0}
\, \bm{v}_{t_1}\, \bm{v}_{t_0^2}\, \bm{v}_{t_0^3}\Vert.
\end{equation}

Finally, if we divide Eq. \eqref{ba3} and Eq. \eqref{ba4}, assuming that the determinants in the denominators do not vanish, $\mbox{det}(M)$ is canceled out and we obtain 
\begin{equation}\label{thezero}
\dfrac{\Vert \bm{u}_{t_0^4}
\, \bm{u}_{t_1}\, \bm{u}_{t_0^2}\, \bm{u}_{t_0^3} \Vert}{\Vert\bfu_{t_0}\, \bfu_{t_1}\, \bfu_{t_0^2}\, \bfu_{t_0^3}\Vert}=\dfrac{\Vert \bm{v}_{t_0^4}
\, \bm{v}_{t_1}\, \bm{v}_{t_0^2}\, \bm{v}_{t_0^3} \Vert}{\Vert\bfv_{t_0}\, \bfv_{t_1}\, \bfv_{t_0^2}\, \bfv_{t_0^3}\Vert}
\end{equation}

Thus, the rational expression at the left of Eq. \eqref{thezero}, which we denote as 
\begin{equation}\label{thefirst}
I_1(\bfu)=\dfrac{\Vert \bm{u}_{t_0^4}
\, \bm{u}_{t_1}\, \bm{u}_{t_0^2}\, \bm{u}_{t_0^3} \Vert}{\Vert\bfu_{t_0}\, \bfu_{t_1}\, \bfu_{t_0^2}\, \bfu_{t_0^3}\Vert},
\end{equation}
is invariant under projectivities, because $I_i(\bfu)=I_i(M\bfu)$ for any non-singular $4\times 4$ matrix $M$; hence, $M\bfu=\bfv$ implies that $I_1(\bfu)=I_1(\bfv)$. Other choices for $(k,l)$ give rise to other invariants, some of which we will be using in the next sections. 

If we have two parametrizations $\bfu,\bfv$ and several invariants $I_i$, $i=1,\ldots,m$, $I_i(\bfu)=I_i(\bfv)$ are necessary conditions for the equality $M\bfu=\bfv$ to hold. In the next section we will see that a good choice of invariants can ensure that these conditions are also sufficient. 

\begin{remark} Certainly, from Theorem \ref{theo0.1} we observe that what we want to recognize is not a relationship like $M\bfu=\bfv$, which could be addressed directly, but $M\bfu=\bfv\circ\varphi$, with $\varphi$ an unknown M\"obius transformation. It is the presence of $\varphi$ that makes the problem more difficult. But even though other tools will also be required, as we will see in the next section the notion of invariant addressed here will be crucial. 
\end{remark}
}
  


\section{A new method to detect projective equivalence.}\label{newmeth}

In this section we consider two curves $\bm{C_1},\bm{C_2}$, defined by homogeneous parametrizations $\bm{p}$ and $\bm{q}$, satisfying the assumptions in Subsection \ref{subsec1}.

\subsection{Overall strategy.}\label{subsec3.1}

As in previous approaches, our strategy takes advantage of Theorem \ref{theo0.1}, and proceeds by first computing the M\"obius transformation $\varphi$ in the statement of Theorem \ref{theo0.1}. If no such transformation is found, the curves $\bm{C_1},\bm{C_2}$ are not projectively equivalent. Otherwise, the matrix $M$ defining the projectivity between $\bm{C_1},\bm{C_2}$ is determined from $\varphi$; we will see that in our case this just amounts to performing a matrix multiplication.  

The main ideas in our approach, which we will develop in order later, are the following. 

\begin{itemize}
\item [(A)] {\it Invariant functions.} We start with four {rational expressions in terms of the components of a parametrization and its derivatives}, that we denote by $I_i$, $i\in \{1,2,3,4\}$. These {expressions} are defined as quotients of certain determinants, {and are invariant under projectivities in the sense of Subsection \ref{subsec2.2}; their 
invariance can be proven with the same strategy used in Subsection \ref{subsec2.2}, and follows from the property for determinants recalled in Eq. \eqref{ref-det}.} From Theorem \ref{theo0.1}, if $\bm{p},\bm{q}$ correspond to projectively equivalent curves then $M\bm{p}=\bm{q}\circ\varphi$, with $\varphi$ a M\"obius transformation. {Then, using the notion of invariance provided in Subsection \ref{subsec2.2}}, we have 
\begin{equation}\label{neweq}
I_i(\bm{p})=I_i(\bm{q}\circ \varphi)
\end{equation}
for $i\in \{1,2,3,4\}$. These are necessary conditions for $\bm{C_1},\bm{C_2}$ to be projectively equivalent, and will be revealed to be also sufficient. 

The equations that stem from Eq. \eqref{neweq} would give rise to a polynomial system in the parameters of $\varphi$. However, this system has a high order, and therefore solving it implies a high computational cost that we want to avoid. 

\item [(B)] {\it Projective curvatures.} In order to avoid solving a big polynomial system, we will derive, from the $I_i$, two more rational expressions $\kappa_1,\kappa_2$, {also invariant in the sense of Subsection \ref{subsec2.2}}, that we will refer to as \emph{projective curvatures}. Since $\kappa_1,\kappa_2$ are also invariants, $\kappa_1,\kappa_2$ do satisfy that
\begin{equation}\label{neweq2}
\kappa_i(\bm{p})=\kappa_i(\bm{q}\circ \varphi)
\end{equation}
for $i=1,2$. But the advantage of the $\kappa_i$ is that while in general $I_i(\bm{q}\circ \varphi)\neq I_i(\bm{q})\circ \varphi$, the $\kappa_i$ do satisfy $\kappa_i(\bm{q}\circ \varphi)=\kappa_i(\bm{q})\circ \varphi${; because of this, we say that the $\kappa_i$ {\it commute} with M\"obius transformations, while the $I_i$ do not}. By taking together the two relationships in Eq. \eqref{neweq2} for $i=1,2$ we can find the whole $\varphi$ as a special quadratic factor of the gcd of two polynomials built from the $\kappa_i$. Thus, to compute $\varphi$ we just need {to employ} gcd computing and factoring, and we do not need to solve any polynomial system. This idea was inspired by the strategy in \citet{Alcazar201551} to compute the symmetries of a rational space curve, where the classical curvature and torsion are used in a similar way. {The term {\it projective curvatures} refers to the fact that the $\kappa_i$ play here a role similar to the role played by the classical curvature and torsion to identify a curve up to rigid motions.}

\item [(C)] {\it Projective equivalences.} Once $\varphi$ is obtained, the nonsingular matrix $M$ defining the projective equivalence can be computed. This just requires performing matrix multiplications involving the parametrizations $\bfp$ and $\bfq\circ \varphi$. 
\end{itemize}

\subsection{{Invariants}.}\label{subsec3.2}

In the rest of the paper, we will use the notation $\Vert\bfw_1\,\ldots\,\bfw_n\Vert$ for the determinant of $n$ vectors $\bfw_i\in {\Bbb R}^n$, and the notation $[\bfw_1\,\ldots\,\bfw_n]$ for the $n\times n$ matrix whose columns are the $\bfw_i$. Additionally, we will use the notation for partial derivatives introduced in Eq. \eqref{notation}. 

In order to motivate {the rational expressions we are going to introduce}, we consider first two homogeneous parametrizations $\bfu(t_0,t_1),\bfv(t_0,t_1)$ of curves of degree $n$ in ${{\Bbb P}^3(\mathbb{R})}$ defining two projectively equivalent curves, such that 
\begin{equation}\label{Meq}
M\bfu(t_0,t_1)=\bfv(t_0,t_1)
\end{equation}
where $M$ represents a projectivity; notice that because of Theorem \ref{theo0.1}, what we are pursuing is exactly Eq. \eqref{Meq}, with $\bfu:=\bfp$, and $\bfv:=\bfq\circ \varphi$ (where $\varphi$ is unknown). Now let $D(\bfu)(t_0,t_1),D(\bfv)(t_0,t_1)$ be the matrices defined as
\begin{equation}\label{Dmatrices}
D(\bfu)=[ \bfu_{t_0}\, \bfu_{t_1}\, \bfu_{t_0^2}\, \bfu_{t_0^3}],\quad D(\bfv)=[ \bfv_{t_0}\, \bfv_{t_1}\, \bfv_{t_0^2}\, \bfv_{t_0^3}].
\end{equation}
Because of Eq. \eqref{Meq}, one can see that $M\cdot D(\bfu)=D(\bfv)$. Assume that the determinant of $D(\bfu)$ is not identically zero, i.e. that $\Vert \bfu_{t_0}\, \bfu_{t_1}\, \bfu_{t_0^2}\, \bfu_{t_0^3}\Vert$ does not vanish identically; later, in Lemma \ref{lem9}, we will see that this holds in our case. Then 
\begin{equation}\label{MD}
M=D(\bfv)(D(\bfu))^{-1}.
\end{equation}
Not let us differentiate $D(\bfv)(D(\bfu))^{-1}$ with respect to $t_k$, $k=0,1$, we get that
\begin{equation}\label{long}
\begin{split}
\dfrac{\partial (D(\bm{v})\cdot (D(\bm{u}))^{-1})}{\partial t_k} &= \dfrac{\partial D(\bm{v})}{\partial t_k}\cdot (D(\bm{u}))^{-1} +D(\bm{v})\cdot \dfrac{\partial (D(\bm{u}))^{-1}}{\partial t_k} \\
& =\dfrac{\partial D(\bm{v})}{\partial t_k}\cdot (D(\bm{u}))^{-1}-D(\bm{v})(D(\bm{u}))^{-1}\cdot \dfrac{\partial D(\bm{u})}{\partial t_k}\cdot (D(\bm{u}))^{-1} \\
& =D(\bm{v})\cdot \left((D(\bm{v}))^{-1}\cdot\dfrac{\partial D(\bm{v})}{\partial t_k}-(D(\bm{u}))^{-1}\cdot\dfrac{\partial D(\bm{u})}{\partial t_k}\right)\cdot (D(\bm{u}))^{-1}.
\end{split}
\end{equation}
Since $M=D(\bfv)(D(\bfu))^{-1}$, {and $M$ is a constant matrix,} we get that the matrices defined by the derivatives at the left-hand side of the above expression are identically zero, and therefore that 
\begin{equation}\label{eq55}
(D(\bm{u}))^{-1}\cdot\dfrac{\partial D(\bm{u})}{\partial t_k}=(D(\bm{v}))^{-1}\cdot\dfrac{\partial D(\bm{v})}{\partial t_k}
\end{equation}
for $k=0,1$. 

We need a closer look at the matrices $U_k,V_k$, defined as 
\begin{equation}\label{UVk}
U_k=(D(\bm{u}))^{-1}\cdot\dfrac{\partial D(\bm{u})}{\partial t_k},\mbox{ }V_k=(D(\bm{v}))^{-1}\cdot\dfrac{\partial D(\bm{v})}{\partial t_k}.
\end{equation}


\ifnum \value{num}=1 {Notice that, according to Eq. \eqref{eq55}, $U_k=V_k$. Now {performing elementary but lengthy calculations, one can check that (see \cite[Section 3.2.1]{Gozutok2022} for a detailed deduction)},
\small 
\begin{equation}\label{Uk0}
\begin{array}{cc}
U_0=\begin{bmatrix}
0 & \frac{n-1}{t_1} & 0 & \dfrac{\Vert \bm{u}_{t_0^4}
\, \bm{u}_{t_1}\, \bm{u}_{t_0^2}\, \bm{u}_{t_0^3} \Vert}{\Vert\bfu_{t_0}\, \bfu_{t_1}\, \bfu_{t_0^2}\, \bfu_{t_0^3}\Vert} \\
0 & 0 & 0 & \dfrac{\Vert \bm{u}_{t_0} \, \bm{u}_{t_0^4}\, \bm{u}_{t_0^2}\, \bm{u}_{t_0^3} \Vert}{\Vert\bfu_{t_0}\, \bfu_{t_1}\, \bfu_{t_0^2}\, \bfu_{t_0^3}\Vert}\\
1 & -\frac{t_0}{t_1} & 0 & \dfrac{\Vert \bm{u}_{t_0}\, \bm{u}_{t_1}\, \bm{u}_{t_0^4}\, \bm{u}_{t_0^3} \Vert}{\Vert\bfu_{t_0}\, \bfu_{t_1}\, \bfu_{t_0^2}\, \bfu_{t_0^3}\Vert}\\
0 & 0 & 1 & \dfrac{\Vert \bm{u}_{t_0}\, \bm{u}_{t_1}\, \bm{u}_{t_0^2}\, \bm{u}_{t_0^4}\Vert}{\Vert\bfu_{t_0}\, \bfu_{t_1}\, \bfu_{t_0^2}\, \bfu_{t_0^3}\Vert}
\end{bmatrix}, & 
U_1=\begin{bmatrix}
\frac{n-1}{t_1} & -\frac{(n-1)t_0}{t_1^2} & 0 & -\frac{t_0}{t_1} \dfrac{\Vert \bm{u}_{t_0^4}
\, \bm{u}_{t_1}\, \bm{u}_{t_0^2}\, \bm{u}_{t_0^3} \Vert}{\Vert\bfu_{t_0}\, \bfu_{t_1}\, \bfu_{t_0^2}\, \bfu_{t_0^3}\Vert}\\
0 & \frac{n-1}{t_1} & 0 & -\frac{t_0}{t_1}\dfrac{\Vert \bm{u}_{t_0} \, \bm{u}_{t_0^4}\, \bm{u}_{t_0^2}\, \bm{u}_{t_0^3} \Vert}{\Vert\bfu_{t_0}\, \bfu_{t_1}\, \bfu_{t_0^2}\, \bfu_{t_0^3}\Vert}\\
-\frac{t_0}{t_1} & \frac{t_0^2}{t_1^2} & \frac{n-2}{t_1} & -\frac{t_0}{t_1}\dfrac{\Vert \bm{u}_{t_0}\, \bm{u}_{t_1}\, \bm{u}_{t_0^4}\, \bm{u}_{t_0^3} \Vert}{\Vert\bfu_{t_0}\, \bfu_{t_1}\, \bfu_{t_0^2}\, \bfu_{t_0^3}\Vert}\\
0 & 0 & -\frac{t_0}{t_1} & \frac{n-3}{t_1}-\frac{t_0}{t_1}\dfrac{\Vert \bm{u}_{t_0}\, \bm{u}_{t_1}\, \bm{u}_{t_0^2}\, \bm{u}_{t_0^4}\Vert}{\Vert\bfu_{t_0}\, \bfu_{t_1}\, \bfu_{t_0^2}\, \bfu_{t_0^3}\Vert}
\end{bmatrix}
\end{array}
\end{equation}
\normalsize
The expressions for $V_0,V_1$ are the same, after replacing $\bfu$ by $\bfv$. In the fourth column of $U_0$ we observe four quotients of determinants, the first of them being the quotient in Eq. \eqref{thefirst} (see Subsection \ref{subsec2.2}), which also arise in the fourth column of $V_0$. All these quotients {correspond to rational expressions which are invariant under projectivities in the sense of Subsection \ref{subsec2.2}}: this follows from the fact that $U_k=V_k$ for $k=0,1$, but also from the arguments in Subsection \ref{subsec2.2}. These four quotients are the {invariants} that we are going to use.} \else {
\subsubsection{Demonstrating $U_k=V_k$.}

In this subsection, let us show $U_k=V_k$ for $k=0,1$.

First let us show that for all $k=0,1$,
\begin{equation}\label{eq0.3}
(D(\bm{p}))^{-1}\cdot\dfrac{\partial D(\bm{p})}{\partial t_k}=(D(\bm{q}))^{-1}\cdot\dfrac{\partial D(\bm{q})}{\partial t_k},
\end{equation}
where  $(D(\bm{p}))^{-1}$ denotes the inverse matrix of $D(\bm{p})$.

Let $(D(\bm{p}))^{-1}\cdot\dfrac{\partial D(\bm{p})}{\partial t_k}=U_k$ for an unknown matrix $U_k$, $k=0,1$. Then $D(\bm{p})\cdot U_k=\dfrac{\partial D(\bm{p})}{\partial t_k}$. Denote the $j$th column of the matrix $D(\bm{p})$ by $D(\bm{p})_j$, and similarly denote the $j$th column of the matrix $U_k$ by $U_k^j$ for $1\leq j\leq 4$. So we have $8$ systems each of which corresponds to a pair of the values $j$ and $k$, namely
\begin{equation}\label{eq0.4}
D(\bm{p})\cdot U_k^j=\dfrac{\partial D(\bm{p})_j}{\partial t_k},
\end{equation}
for $k=0,1$ and $1\leq j\leq 4$. The system corresponding to each pair is linear in the components of $U_k^j$. On the other hand, since the coefficient matrix $D(\bm{p})$ of each system is non singular ($\Delta(\bm{p})\neq 0$), each system has only one solution. The solution to the systems are
\begin{equation}\label{eq0.5}
U_k^j=\begin{bmatrix}
\dfrac{\lVert \dfrac{\partial D(\bm{p})_j}{\partial t_k}\, \bm{p}_{t_1}\, \bm{p}_{t_0^2}\, \bm{p}_{t_0^3} \rVert}{\Delta(\bm{p})} \\
\dfrac{\lVert \bm{p}_{t_0} \, \dfrac{\partial D(\bm{p})_j}{\partial t_k}\, \bm{p}_{t_0^2}\, \bm{p}_{t_0^3} \rVert}{\Delta(\bm{p})} \\
\dfrac{\lVert \bm{p}_{t_0}\, \bm{p}_{t_1}\, \dfrac{\partial D(\bm{p})_j}{\partial t_k}\, \bm{p}_{t_0^3} \rVert}{\Delta(\bm{p})} \\
\dfrac{\lVert \bm{p}_{t_0}\, \bm{p}_{t_1}\, \bm{p}_{t_0^2}\, \dfrac{\partial D(\bm{p})_j}{\partial t_k} \rVert}{\Delta(\bm{p})} \\
\end{bmatrix}.
\end{equation} 

Using Euler's homogeneous function theorem, we conclude that for $k=0$ and $j<4$,
\begin{equation}\label{eq0.6}
U_0^1=\begin{bmatrix}
0 \\
0 \\
1 \\
0
\end{bmatrix},\,\, U_0^2=\begin{bmatrix}
\frac{n-1}{t_1} \\
0 \\
-\frac{t_0}{t_1} \\
0
\end{bmatrix},\,\, U_0^3=\begin{bmatrix}
0 \\
0\\
0\\
1
\end{bmatrix}
\end{equation}

and for $k=1$ and $j<4$,

\begin{equation}\label{eq0.7}
U_1^1=\begin{bmatrix}
\frac{n-1}{t_1} \\
0 \\
-\frac{t_0}{t_1} \\
0
\end{bmatrix},\,\, U_1^2=\begin{bmatrix}
-\frac{(n-1)t_0}{t_1^2} \\
\frac{n-1}{t_1} \\
\frac{t_0^2}{t_1^2} \\
0
\end{bmatrix},\,\, U_1^3=\begin{bmatrix}
0 \\
0\\
\frac{n-2}{t_1}\\
-\frac{t_0}{t_1}
\end{bmatrix}.
\end{equation}

It is seen that for $k=0,1$ and $j<4$, $U_k^j$ do not depend on $p$. Now for $k=0$ and $j=4$, we have $\dfrac{\partial D(\bm{p})_4}{\partial t_0}=\dfrac{\partial \bm{p}_{t_0^3}}{\partial t_0}=\bm{p}_{t_0^4}$. It is obtained
\begin{equation}\label{eq0.8}
U_0^4=\begin{bmatrix}
\dfrac{\lVert \bm{p}_{t_0^4}
\, \bm{p}_{t_1}\, \bm{p}_{t_0^2}\, \bm{p}_{t_0^3} \rVert}{\Delta(\bm{p})} \\
\dfrac{\lVert \bm{p}_{t_0} \, \bm{p}_{t_0^4}\, \bm{p}_{t_0^2}\, \bm{p}_{t_0^3} \rVert}{\Delta(\bm{p})} \\
\dfrac{\lVert \bm{p}_{t_0}\, \bm{p}_{t_1}\, \bm{p}_{t_0^4}\, \bm{p}_{t_0^3} \rVert}{\Delta(\bm{p})} \\
\dfrac{\lVert \bm{p}_{t_0}\, \bm{p}_{t_1}\, \bm{p}_{t_0^2}\, \bm{p}_{t_0^4} \rVert}{\Delta(\bm{p})} \\
\end{bmatrix}=\begin{bmatrix}
I_1(\bm{p}) \\
I_2(\bm{p})\\
I_3(\bm{p})\\
I_4(\bm{p})
\end{bmatrix}.
\end{equation}
Similarly, for $k=1$ and $j=4$, we have $\dfrac{\partial D(\bm{p})_4}{\partial t_1}=\dfrac{\partial \bm{p}_{t_0^3}}{\partial t_1}=\bm{p}_{t_0^3t_1}=\frac{n-3}{t_1}\bm{p}_{t_0^3}-\dfrac{t_0}{t_1}\bm{p}_{t_0^4}$. It is obtained

\begin{equation}\label{eq0.9}
U_1^4=\begin{bmatrix}
\dfrac{\lVert \frac{n-3}{t_1}\bm{p}_{t_0^3}-\dfrac{t_0}{t_1}\bm{p}_{t_0^4}
\, \bm{p}_{t_1}\, \bm{p}_{t_0^2}\, \bm{p}_{t_0^3} \rVert}{\Delta(\bm{p})} \\
\dfrac{\lVert \bm{p}_{t_0} \, \frac{n-3}{t_1}\bm{p}_{t_0^3}-\dfrac{t_0}{t_1}\bm{p}_{t_0^4}\, \bm{p}_{t_0^2}\, \bm{p}_{t_0^3} \rVert}{\Delta(\bm{p})} \\
\dfrac{\lVert \bm{p}_{t_0}\, \bm{p}_{t_1}\, \frac{n-3}{t_1}\bm{p}_{t_0^3}-\dfrac{t_0}{t_1}\bm{p}_{t_0^4}\, \bm{p}_{t_0^3} \rVert}{\Delta(\bm{p})} \\
\dfrac{\lVert \bm{p}_{t_0}\, \bm{p}_{t_1}\, \bm{p}_{t_0^2}\, \frac{n-3}{t_1}p_{t_0^3}-\dfrac{t_0}{t_1}\bm{p}_{t_0^4} \rVert}{\Delta(\bm{p})} \\
\end{bmatrix}=\begin{bmatrix}
-\frac{t_0}{t_1} I_1(\bm{p}) \\
-\frac{t_0}{t_1} I_2(\bm{p})\\
-\frac{t_0}{t_1} I_3(\bm{p})\\
\frac{n-3}{t_1}-\frac{t_0}{t_1} I_4(\bm{p})
\end{bmatrix}.
\end{equation}

Now let $(D(\bm{q}))^{-1}\cdot\dfrac{\partial D(\bm{q})}{\partial t_k}=V_k$ for an unknown matrix $V_k$, $k=0,1$. Then $D(\bm{q})\cdot V_k=\dfrac{\partial D(\bm{q})}{\partial t_k}$. Denote the $j$th column of the matrix $D(\bm{q})$ by $D(\bm{q})_j$, and similarly denote the $j$th column of the matrix $V_k$ by $V_k^j$ for $1\leq j\leq 4$. Similarly, we again have $8$ systems each of which corresponds to a pair of the values $j$ and $k$ for $q$. For the solutions $V_k^j$ we have $U_k^j=V_k^j$ for all $k$ and $j<4$, since $U_k^j$ and $V_k^j$ do not depend on the parametrizations. In addition, similar operations lead to
\begin{equation}\label{eq0.10}
V_0^4=\begin{bmatrix}
I_1(\bm{q}) \\
I_2(\bm{q})\\
I_3(\bm{q})\\
I_4(\bm{q})
\end{bmatrix},
\end{equation}

and 

\begin{equation}\label{eq0.11}
V_1^4=\begin{bmatrix}
-\frac{t_0}{t_1} I_1(\bm{q}) \\
-\frac{t_0}{t_1} I_2(\bm{q})\\
-\frac{t_0}{t_1} I_3(\bm{q})\\
\frac{n-3}{t_1}-\frac{t_0}{t_1} I_4(\bm{q})
\end{bmatrix}.
\end{equation}

Since $I_i(\bm{p})=I_i(\bm{q})$ for all $1\leq i\leq 4$, $U_k^4=V_k^4$ for all $k$. Therefore $U_k=V_k$ for all $k$.
} \fi

Thus, let us denote 
\begin{align}
&A_1(\bm{u}):=\Vert \bm{u}_{t_0^4}\, \bm{u}_{t_1}\, \bm{u}_{t_0^2}\, \bm{u}_{t_0^3} \Vert,A_2(\bm{u}):=\Vert \bm{u}_{t_0}\, \bm{u}_{t_0^4}\, \bm{u}_{t_0^2}\, \bm{u}_{t_0^3} \Vert, A_3(\bm{u}):=\Vert \bm{u}_{t_0}\, \bm{u}_{t_1}\, \bm{u}_{t_0^4}\, \bm{u}_{t_0^3} \Vert, \nonumber \\
&A_4(\bm{u}):=\Vert \bm{u}_{t_0}\, \bm{u}_{t_1}\, \bm{u}_{t_0^2}\, \bm{u}_{t_0^4} \Vert, \Delta(\bm{u}):=\Vert \bm{u}_{t_0}\, \bm{u}_{t_1}\, \bm{u}_{t_0^2}\, \bm{u}_{t_0^3} \Vert. \label{As}
\end{align}

Then we {have the} following four rational expressions, which are the expressions arising in Eq. \eqref{Uk0}:
\begin{equation}\label{first-inv}
I_1(\bm{u}):=\dfrac{A_1(\bm{u})}{\Delta(\bm{u})},\, I_2(\bm{u}):=\dfrac{A_2(\bm{u})}{\Delta(\bm{u})},\, I_3(\bm{u}):=\dfrac{A_3(\bm{u})}{\Delta(\bm{u})},\, I_4(\bm{u}):=\dfrac{A_4(\bm{u})}{\Delta(\bm{u})}.
\end{equation}

The following lemma, proven in \ref{AppendixA}, guarantees that under our hypotheses the {rational expressions in Eq. \eqref{first-inv} are well defined}. 

\begin{lemma}\label{lem9}
Let $\bm{C}$ in ${{\Bbb P}^3(\mathbb{R})}$ be a rational algebraic curve properly parametrized by $\bm{u}$, satisfying the hypotheses (i-iv) in Subsection \ref{subsec1}. Then $\Delta(\bm{u})$ is not identically zero. 
\end{lemma}

\vspace{0.3 cm}
Now let us go back to our curves $\bm{C_1},\bm{C_2}$, defined by homogeneous parametrizations $\bm{p}$ and $\bm{q}$ of degree $n\geq 4$, as in Subsection \ref{subsec1}. We recall that we are assuming that $\bm{p},\bm{q}$ satisfy hypotheses (i-iv) in Subsection \ref{subsec1}. {The fact that the expressions in Eq. \eqref{first-inv} are invariant in the sense of Subsection \ref{subsec2.2} implies that $I_i(\bfp)=I_i(M\bfp)=I_i(\bfq\circ \varphi)$ for $i=1,2,3,4$. Thus, the existence of a M\"obius transformation $\varphi$ such that $I_i(\bfp)=I_i(\bfq\circ \varphi)$ for $i=1,2,3,4$ is a necessary condition for $\bm{C_1},\bm{C_2}$ to be projective equivalent. However, the next theorem proves that in this case, these conditions are also sufficient, which is a consequence of how these invariants were chosen. The intuitive idea is as follows: the conditions $I_i(\bfp)=I_i(\bfq\circ \varphi)$ ensure that Eq. \eqref{UVk}, with $\bfp,\bfq\circ \varphi$ replacing $\bfu,\bfv$, is satisfied. From here, Eq. \eqref{eq55} is also satisfied, and therefore the derivatives at the left-hand side of Eq. \eqref{long} are zero. Thus, Eq. \eqref{MD} must be satisfied by some constant matrix $M$, with $\bfp,\bfq\circ \varphi$ replacing $\bfu,\bfv$. Theorem \ref{theo0.1} makes the rest.}

\begin{theorem}\label{teo29}
Let $\bm{C}_1,\bm{C}_2$ be two rational algebraic curves properly parametrized by $\bm{p},\bm{q}$ satisfying hypotheses $(i-iv)$. Then $\bm{C}_1,\bm{C}_2$ are projectively equivalent if and only if there exists a M\"obius transformation $\varphi$ such that 
\begin{equation}\label{Ipqmob}
I_i(\bfp)=I_i(\bfq\circ \varphi)
\end{equation}
for $i\in\{1,2,3,4\}$. 
\end{theorem}

\begin{proof} The implication $(\Rightarrow)$ follows from Theorem \ref{theo0.1} and the discussion at the beginning of this subsection. So let us focus on $(\Leftarrow)$. Let $\bfu:=\bfp$, $\bfv:=\bfq\circ \varphi$. Since by hypothesis $I_i(\bfu)=I_i(\bfv)$ for $i=1,2,3,4$, taking Eq. \eqref{Uk0} into account we have $U_k=V_k$ for $k=0,1$, so Eq. \eqref{eq55} holds for $k=0,1$; notice that since the determinants of $D(\bfu), D(\bfv)$ are, precisely, $\Delta(\bfu),\Delta(\bfv)$, by Lemma \ref{lem9} the inverses $D(\bfu)^{-1}, D(\bfv)^{-1}$ exist. Therefore, the matrix $D(\bm{v})\cdot (D(\bm{u}))^{-1}$ is a constant nonsingular matrix $M$. Thus, $M\cdot D(\bm{u})=D(\bm{v})$, so $M \bm{u}_{t_0}=\bm{v}_{t_0}$ and $M \bm{u}_{t_1}=\bm{v}_{t_1}$. Using Euler's Homogeneous Function Theorem, we have 
\begin{equation*}
n\bm{v}=t_0\bm{v}_{t_0}+t_1\bm{v}_{t_1}=t_0M \bm{v}_{t_0}+t_1M \bm{v}_{t_1}=M (t_0\bm{u}_{t_0}+t_1\bm{u}_{t_1})=nM \bm{u},
\end{equation*}
so $M \bm{u}=\bm{v}$. 
\end{proof}

The relationships in Eq. \eqref{Ipqmob} lead to a polynomial system in the parameters of the M\"obius transformation $\varphi${, which we might try to solve to find $\varphi$}. However, this system has a high degree. Because of this, we will derive other {rational expressions, also invariant,} that we call \emph{projective curvatures} {and that will allow us to solve our problem without using polynomial system solving}. This is done in the next subsection. 

\subsection{Projective curvatures.} \label{subsec-proj-curv}

A first question when examining the relationships in Eq. \eqref{Ipqmob} is how the $I_i(\bfq)$ change when $\bfq$ is composed with $\varphi$. Writing $\varphi(t_0,t_1)=(at_0+bt_1,ct_0+dt_1)=(u,v)$, calling $\delta=ad-bc\neq 0$ and using the Chain Rule, we get that

\begin{align}
v^4 I_1(\bm{q}\circ\varphi) &=c^3(n-1)(n-2)(n-3)(3v+dt_1)+c^2(n-1)(n-2)\delta t_1(2v+dt_1)I_4(\bm{q})\circ\varphi\nonumber \\
& -c(n-1)\delta ^2t_1^2(v+dt_1)I_3(\bm{q})\circ\varphi+\delta ^3t_1^4(dI_1(\bm{q})\circ\varphi-bI_2(\bm{q})\circ\varphi)\label{eq17} \\ 
v^4I_2(\bm{q}\circ\varphi) &=-c^4(n-1)(n-2)(n-3)t_1-c^3(n-1)(n-2)\delta t_1^2I_4(\bm{q})\circ\varphi\nonumber \\
& +c^2(n-1)\delta ^2t_1^3I_3(\bm{q})\circ\varphi+\delta ^3t_1^4(aI_2(\bm{q})\circ\varphi-cI_1(\bm{q})\circ\varphi)\label{eq18} \\ 
v^2I_3(\bm{q}\circ\varphi) &=-6c^2(n-2)(n-3)-3c(n-2)\delta t_1I_4(\bm{q})\circ\varphi+\delta ^2t_1^2I_3(\bm{q})\circ\varphi\label{eq19} \\
vI_4(\bm{q}\circ\varphi) &=4c(n-3)+\delta t_1I_4(\bm{q})\circ\varphi \label{eq20}.
\end{align}

From these expressions we see that the $I_i$ do not commute with $\varphi$, i.e. in general $I_i(\bfq\circ \varphi)\neq I_i(\bfq)\circ \varphi$. We aim to find {rational functions} that do commute with $\varphi$. In order to do that, we substitute the equations \eqref{eq17}-\eqref{eq20} into the relationships in Eq. \eqref{Ipqmob}, and we get
\begin{align}
v^4 I_1(\bm{p})(t_0,t_1) &=c^3(n-1)(n-2)(n-3)(3v+dt_1)+c^2(n-1)(n-2)\delta t_1(2v+dt_1)I_4(\bm{q}(u,v)\nonumber \\
& -c(n-1)\delta ^2t_1^2(v+dt_1)I_3(\bm{q})(u,v)+\delta ^3t_1^4(dI_1(\bm{q})(u,v)-bI_2(\bm{q})(u,v))\label{eq23} \\ 
v^4I_2(\bm{p})(t_0,t_1) &=-c^4(n-1)(n-2)(n-3)t_1-c^3(n-1)(n-2)\delta t_1^2I_4(\bm{q})(u,v)\nonumber \\
& +c^2(n-1)\delta ^2t_1^3I_3(\bm{q})(u,v)+\delta ^3t_1^4(aI_2(\bm{q})(u,v)-cI_1(\bm{q})(u,v))\label{eq24} \\ 
v^2I_3(\bm{p})(t_0,t_1) &=-6c^2(n-2)(n-3)-3c(n-2)\delta t_1I_4(\bm{q})(u,v)+\delta ^2t_1^2I_3(\bm{q})(u,v)\label{eq25} \\
vI_4(\bm{p})(t_0,t_1) &=4c(n-3)+\delta t_1I_4(\bm{q})(u,v) \label{eq26},
\end{align}
{where we have performed the substitution $u:=at_0+bt_1$, $v:=ct_0+dt_1$ for the components of the M\"obius transformation $\varphi$ (see Eq. \eqref{eq-moeb}). At this point $u,v$ are considered as independent variables, i.e. we forget about the dependence between $u,v$ and $t_0,t_1$.}

{Next, we} want to eliminate the parameters $a,b,c,d$ from these equations, something that we can {translate} in terms of polynomial ideals. Indeed, let us write $J_i=I_i(\bm{q})(u,v)$ for $i\in \{1,2,3,4\}$, and let us denote the $I_i(\bm{p})$ by just $I_i$. Then after clearing denominators, the equations \eqref{eq23}-\eqref{eq26} generate a polynomial ideal ${\mathcal I}$ of 
\[
{\Bbb R}[a,b,c,d,t_1,v,I_1,I_2,I_3,I_4,J_1,J_2,J_3,J_4].
\]
Eliminating $a,b,c,d$ from equations \eqref{eq23}-\eqref{eq26} amounts to finding elements in the elimination ideal 
\[
{\mathcal I}^\star={\mathcal I}\cap{\Bbb R}[t_1,v,I_1,I_2,I_3,I_4,J_1,J_2,J_3,J_4].
\]


\ifnum \value{num}=1 {
In our case, this can be done by hand, without using Gr\"obner bases; the process consists of several easy, but lengthy, substitutions and manipulations \cite[see][Section 3.3.1]{Gozutok2022}. Eventually we get the expressions} 
\else {
\subsubsection{Eliminating variables.}

In our case, this can be done by hand, without using Gr\"obner bases; the process consists of several easy, but lengthy, following substitutions and manipulations.

Now we are ready to eliminate the coefficients $a,b,c,d$ from the above system to find a system such that $\left\lbrace k_1(\bm{p})=k_1(\bm{q})(u,v),k_1(\bm{p})=k_1(\bm{q})(u,v)\right\rbrace$. We first try to eliminate $a,b,d$ from \eqref{eq23} and \eqref{eq24}, then multiplying \eqref{eq23} by $t_0$ and \eqref{eq24} by $-t_1$ and summing the results, we have

\begin{align}
v^4I_0(\bm{p})(t_0,t_1)&=4c^3(n-1)(n-2)(n-3)t_1v+3c^2(n-1)(n-2)\delta t_1^2vI_4(\bm{q})(u,v) \nonumber\\
&-2c(n-1)\delta ^2t_1^3vI_3(\bm{q})(u,v)+\delta ^3t_1^4I_0(\bm{q})(u,v), \label{eq32}
\end{align}
where $I_0(\bm{p})(t_0,t_1)=t_1I_1(\bm{p})(t_0,t_1)-t_0I_2(\bm{p})(t_0,t_1)$.

Again to eliminate $a$ from \eqref{eq24}, we use $av-cu=\delta$. Substituting $a=\dfrac{\delta +cu}{v}$ in \eqref{eq24}, it is obtained

\begin{align}
v^5I_2(\bm{p})(t_0,t_1)&=-c^4(n-1)(n-2)(n-3)t_1v-c^3(n-1)(n-2)\delta t_1^2vI_4(\bm{q})(u,v)\nonumber \\
&+c^2(n-1)\delta ^2t_1^3vI_3(\bm{q})(u,v)-c\delta^3t_1^4I_0(\bm{q})(u,v)+\delta^4t_1^5I_2(\bm{q})(u,v).\label{eq33}
\end{align}

We know that the coefficients of the Möbius transformation satisfy $\delta =ad-bc\neq 0$. Then there is a number $s\neq 0$ such that $\delta s=1$. Now we use this in the equations \eqref{eq23}-\eqref{eq26} in order to eliminate the parameters $\delta ,c$. Thus multiplying the equations \eqref{eq26}, \eqref{eq25}, \eqref{eq32}, \eqref{eq33} by $s$, $s^2$, $s^3$, $s^4$, respectively, we have

\begin{align}
svI_4(\bm{p})(t_0,t_1) &=4(cs)(n-3)+t_1I_4(\bm{q})(u,v) \label{eq34}\\
s^2v^2I_3(\bm{p})(t_0,t_1) &=-6(cs)^2(n-2)(n-3)-3(cs)(n-2) t_1I_4(\bm{q})(u,v)+t_1^2I_3(\bm{q})(u,v)\label{eq35} \\
s^3v^4I_0(\bm{p})(t_0,t_1) &=4(cs)^3(n-1)(n-2)(n-3)t_1v+3(cs)^2(n-1)(n-2)t_1^2vI_4(\bm{q})(u,v)\nonumber \\
&-2(cs)(n-1)t_1^3vI_3(\bm{q})(u,v)+t_1^4I_0(\bm{q})(u,v)\label{eq36} \\
s^4v^5I_2(\bm{p})(t_0,t_1) &=-(cs)^4(n-1)(n-2)(n-3)t_1v-(cs)^3(n-1)(n-2)t_1^2vI_4(\bm{q})(u,v)\nonumber\\
&+(cs)^2(n-1)t_1^3vI_3(\bm{q})(u,v)-(cs)t_1^4I_0(\bm{q})(u,v)+t_1^5I_2(\bm{q})(u,v).\label{eq37}
\end{align}

From now on unless otherwise stated explicitly, we denote $I_i(\bm{q})(u,v)$ by $J_i(\bm{q})$ and drop $(t_0,t_1)$ from the function $I_i(\bm{p})(t_0,t_1)$ for the sake of shortening of the equations. We are ready to eliminate $c$ from the above equations. Let us get $cs$ from first equation and write as
\begin{equation*}
cs=\dfrac{vsI_4(\bm{p})-t_1J_4(\bm{q})}{4(n-3)}.
\end{equation*} 

Substituting $cs$ in the equations \eqref{eq35},\eqref{eq36}, \eqref{eq37} and using Theorem \ref{teo17}, we have
\begin{equation}\label{eq38}
s^2=\dfrac{t_1^2(8(n-3)J_3(\bm{q})+3(n-2)J_4^2(\bm{q})}{v^2(8(n-3)I_3(\bm{p})+3(n-2)I_4^2(\bm{p}))},
\end{equation}

\begin{align}
s^3 &=\dfrac{t_1^4(8(n-3)^2J_0(\bm{q})+4(n-1)(n-3)t_1J_3(\bm{q})J_4(\bm{q})+(n-1)(n-2)t_1J_4^3(\bm{q}))}{v^4(8(n-3)^2I_0(\bm{p})+4(n-1)(n-3)t_1I_3(\bm{p})I_4(\bm{p})+(n-1)(n-2)t_1I_4^3(\bm{p}))},\label{eq39}
\end{align}
\small
\begin{align}
s^4 &=\dfrac{t_1^5(256(n-3)^3J_2(\bm{q})+64(n-3)^2J_0(\bm{q})J_4(\bm{q})+16(n-1)(n-3)t_1J_3(\bm{q})J_4^2(\bm{q})+3(n-1)(n-2)t_1J_4^4(\bm{q}))}{v^5(256(n-3)^3I_2(\bm{p})+64(n-3)^2I_0(\bm{p})I_4(\bm{p})+16(n-1)(n-3)t_1I_3(\bm{p})I_4^2(\bm{p})+3(n-1)(n-2)t_1I_4^4(\bm{p}))},\label{eq40}
\end{align}
\normalsize
respectively.

Eliminating $s$ by equating the cube of \eqref{eq38} and the square of \eqref{eq39} it is obtained
\begin{align}
&\dfrac{(8(n-3)^2I_0(\bm{p})+4(n-1)(n-3)t_1I_3(\bm{p})I_4(\bm{p})+(n-1)(n-2)t_1I_4^3(\bm{p}))^2}{t_1^2(8(n-3)I_3(\bm{p})+3(n-2)I_4^2(\bm{p}))^3} \nonumber\\
&=\dfrac{(8(n-3)J_0(\bm{q})+4(n-1)(n-3)t_1J_3(\bm{q})J_4(q)+(n-1)(n-2)t_1J_4^3(\bm{q}))^2}{v^2(8(n-3)J_3(\bm{q})+3(n-2)J_4^2(\bm{q}))^3}. \label{eq41}
\end{align} 

And by equating the square of \eqref{eq38}, and \eqref{eq40} it is obtained

\begin{align}
&\dfrac{256(n-3)^3I_2(\bm{p})+64(n-3)^2I_0(\bm{p})I_4(\bm{p})+16(n-1)(n-3)t_1I_3(\bm{p}))I_4^2(\bm{p})+3(n-1)(n-2)t_1I_4^4(\bm{p})}{t_1(8(n-3)I_3(\bm{p})+3(n-2)I_4^2(\bm{p}))^2}\nonumber\\
&=\dfrac{256(n-3)^3J_2(\bm{q})+64(n-3)^2J_0(\bm{q})J_4(\bm{q})+16(n-1)(n-3)t_1J_3(\bm{q}))J_4^2(\bm{q})+3(n-1)(n-2)t_1J_4^4(\bm{q})}{v(8(n-3)J_3(\bm{q})+3(n-2)J_4^2(\bm{q}))^2}.\label{eq42}
\end{align}

Eventually we get
} \fi
\small
\begin{align}
&\dfrac{(8(n-3)^2I_0(\bm{p})+4(n-1)(n-3)t_1I_3(\bm{p})I_4(\bm{p})+(n-1)(n-2)t_1I_4^3(\bm{p}))^2}{t_1^2(8(n-3)I_3(\bm{p})+3(n-2)I_4^2(\bm{p}))^3} \nonumber\\
&=\dfrac{(8(n-3)J_0(\bm{q})+4(n-1)(n-3)t_1J_3(\bm{q})J_4(q)+(n-1)(n-2)t_1J_4^3(\bm{q}))^2}{v^2(8(n-3)J_3(\bm{q})+3(n-2)J_4^2(\bm{q}))^3}, \label{eq41*}
\end{align} 
\normalsize
and
\small
\begin{align}
&\dfrac{256(n-3)^3I_2(\bm{p})+64(n-3)^2I_0(\bm{p})I_4(\bm{p})+16(n-1)(n-3)t_1I_3(\bm{p})I_4^2(\bm{p})+3(n-1)(n-2)t_1I_4^4(\bm{p})}{t_1(8(n-3)I_3(\bm{p})+3(n-2)I_4^2(\bm{p}))^2}\nonumber\\
&=\dfrac{256(n-3)^3J_2(\bm{q})+64(n-3)^2J_0(\bm{q})J_4(\bm{q})+16(n-1)(n-3)t_1J_3(\bm{q})J_4^2(\bm{q})+3(n-1)(n-2)t_1J_4^4(\bm{q})}{v(8(n-3)J_3(\bm{q})+3(n-2)J_4^2(\bm{q}))^2},\label{eq42*}
\end{align}
\normalsize
where 
\[
I_0(\bfp)= t_1I_1(\bfp)-t_0I_2(\bfp),\mbox{ }J_0(\bfq)= vJ_1(\bfq)-uJ_2(\bfq).
\]

Notice that Eq. \eqref{eq41*} and \eqref{eq42*} have a very special structure: if we examine the right-hand side and the left-hand side of each of these equations, we detect the same function but evaluated at $(t_0,t_1)$, at the left, and at $(u,v)$, at the right. This motivates our definition of the following two functions, that we call \emph{projective curvatures}:

\small
\begin{equation}\label{kappas}
\begin{split}
& \kappa_1(\bm{p})=\dfrac{(8(n-3)^2I_0(\bm{p})+4(n-1)(n-3)t_1I_3(\bm{p})I_4(\bm{p})+(n-1)(n-2)t_1I_4^3(\bm{p}))^2}{t_1^2(8(n-3)I_3(\bm{p})+3(n-2)I_4^2(\bm{p}))^3} \\
& \kappa_2(\bm{p})=\dfrac{256(n-3)^3I_2(\bm{p})+64(n-3)^2I_0(\bm{p})I_4(\bm{p})+16(n-1)(n-3)t_1I_3(\bm{p})I_4^2(\bm{p})+3(n-1)(n-2)t_1I_4^4(\bm{p})}{t_1(8(n-3)I_3(\bm{p})+3(n-2)I_4^2(\bm{p}))^2}.
\end{split}
\end{equation}
\normalsize

\begin{remark} \label{otherproj}
Notice that there are additional possibilities for projective curvatures, other than $\kappa_1,\kappa_2$ in Eq. \eqref{kappas}. What we really want are elements in the ideal ${\mathcal I}^\star$ which correspond to the subtraction of the evaluations of a certain rational function at $t_1,I_1,I_2,I_3,I_4$ and at $v,J_1,J_2,J_3,J_4$, respectively. We do not have yet a complete theoretical explanation of why the ideal ${\mathcal I}^\star$ contains such elements. This probably requires further look into the theory of differential invariants {(see \citep{MR836734,dolgachev_2003,olver_1995,mansfield_2010} for further information on this topic)}.
\end{remark}

The next result follows directly from Eq. \eqref{eq41*} and \eqref{eq42*}. 

\begin{lemma}\label{lem22}
Let $\bm{C}$ be a rational algebraic curve properly parametrized by $\bm{p}$ satisfying hypotheses (i-iv) and let $\varphi(t_0,t_1)=(at_0+bt_1,ct_0+dt_1)$ be a M\"obius transformation with $ad-bc\neq 0$. The following equalities hold. 
\begin{itemize}
\item[i.] $\kappa_1(\bm{p}\circ \varphi)=\kappa_1(\bm{p})\circ \varphi$,
\item[ii.] $\kappa_2(\bm{p}\circ \varphi)=\kappa_2(\bm{p})\circ \varphi$.
\end{itemize}
\end{lemma}

The fact that $\kappa_1,\kappa_2$ are well defined follows from the following result, which is proven in \ref{AppendixB}. In fact, in \ref{AppendixB} we prove a stronger result which implies this lemma, namely that the $I_i$, $i\in\{1,2,3,4\}$, are algebraically independent. 

\begin{lemma} \label{denomnot}
The denominators in $\kappa_1,\kappa_2$ do not identically vanish, and therefore $\kappa_1,\kappa_2$ are well defined. 
\end{lemma}

Now we are ready to present our main result, that characterizes the projective equivalences of rational $3D$ curves in terms of the rational invariant functions $\kappa_1$ and $\kappa_2$.

\begin{theorem}\label{teo23}
Let $\bm{C}_1,\bm{C}_2$ be two rational algebraic curves properly parametrized by $\bm{p},\bm{q}$ satisfying hypotheses (i-iv). Then $\bm{C}_1, \bm{C}_2$ are projectively equivalent if and only if there exist {two linear functions $u:=at_0+bt_1,v:=ct_0+dt_1$, with $ad-bc\neq 0$, satisfying the following equations}
\begin{align}
\kappa_1(\bm{p})(t_0,t_1)&=\kappa_1(\bm{q})(u,v) \label{eq44} \\
\kappa_2(\bm{p})(t_0,t_1)&=\kappa_2(\bm{q})(u,v), \label{eq45}
\end{align}
and such that $D(\bfq \circ \varphi)(D(\bfp))^{-1}${, with $\varphi (t_0,t_1)=(at_0+bt_1,ct_0+dt_1)$,} is a constant matrix $M$. Furthermore, $f(\bfx)=M\cdot\bfx$ is a projective equivalence between $\bm{C}_1, \bm{C}_2$.
\end{theorem}

\begin{proof} $(\Rightarrow)$ From Theorem \ref{theo0.1}, there exists a M\"obius transformation $\varphi$ such that $M\cdot\bfp=\bfq\circ \varphi$; furthermore, from the discussion at the beginning of Subsection \ref{subsec3.2}, $M=D(\bfq \circ \varphi)(D(\bfp))^{-1}$. By Theorem \ref{teo29}, $I_i(\bfp)=I_i(\bfq \circ \varphi)$ for $i\in\{1,2,3,4\}$, and therefore Eq. \eqref{eq41*} and \eqref{eq42*} hold. {Picking $u,v$ as the components of $\varphi$, the} rest follows from the definition of $\kappa_1,\kappa_2$. $(\Leftarrow)$ From the proof of the implication $``\Leftarrow"$ in Theorem \ref{teo29}, if $D(\bfq \circ \varphi)(D(\bfp))^{-1}=M$, with $M$ constant, then $M\cdot\bfp=\bfq\circ \varphi$, so $f(\bfx)=M\cdot\bfx$ is a projective equivalence between $\bm{C}_1,\bm{C}_2$.
\end{proof}


\section{The algorithm.}\label{sec5}

In this section we will see how to turn the result in Theorem \ref{teo23} into an algorithm to detect projective equivalence. \begin{color}{black}Before proceeding, let us provide some insight on the main idea. It is probably clearer to consider the problem in an affine setting. The projective curvatures $\kappa_1,\kappa_2$ introduced in Eq. \eqref{kappas}, in an affine setting, correspond to two rational functions 
\[
\tilde{\kappa}_1(t),\mbox{ }\tilde{\kappa}_2(t),
\]
where now $t$ is an affine parameter. Furthermore, from Theorem \ref{teo23} we know that 
\begin{equation}\label{aux01}
\tilde{\kappa}_1(t)=\tilde{\kappa}_1(\tilde{\varphi}(t)),\mbox{ }\tilde{\kappa}_2(t)=\tilde{\kappa}_2(\tilde{\varphi}(t))
\end{equation}
where 
\begin{equation}\label{aux02}
\tilde{\varphi}(t)=\dfrac{at+b}{ct+d},
\end{equation}
which is also an affine rational function, corresponding to the M\"obius function $\varphi$ in Theorem \ref{teo23}. Now let $s$ be a new variable, and consider the expressions 
\begin{equation}\label{aux11}
\tilde{\kappa}_1(t)-\tilde{\kappa}_2(s)=0,\mbox{ }\tilde{\kappa}_2(t)-\tilde{\kappa}_2(s)=0.
\end{equation}
After clearing denominators, one can see the expressions in Eq. \eqref{aux11} as defining two algebraic curves in the $(t,s)$ plane. Because of Eq. \eqref{aux01}, all the $(t,s)$ points satisfying that $s=\tilde{\varphi}(t)$ are points of these two curves. Thus, these two curves have infinitely many points in common, so by Bezout's Theorem they must share a factor. And this factor is, precisely, the factor obtained from $s-\tilde{\varphi}(t)=0$ after clearing denominators. We just need to formalize this idea, and work projectively.
\end{color}

In order to do this, first we write  
\begin{align}
&\kappa_1(\bm{p})(t_0,t_1)=\dfrac{U(t_0,t_1)}{V(t_0,t_1)} &\kappa_2(\bm{p})(t_0,t_1)=\dfrac{Y(t_0,t_1)}{Z(t_0,t_1)}, \label{eq46} \\
&\kappa_1(\bm{q})(t_0,t_1)=\dfrac{\bar{U}(t_0,t_1)}{\bar{V}(t_0,t_1)} &\kappa_2(\bm{q})(t_0,t_1)=\dfrac{\bar{Y}(t_0,t_1)}{\bar{Z}(t_0,t_1)}, \label{eq47}
\end{align}
where $U,V,Y,Z$ and $\bar{U},\bar{V},\bar{Y},\bar{Z}$ are homogeneous polynomials such that $\gcd(U,V)=1$, $\gcd(Y,Z)=1$, $\gcd(\bar{U},\bar{V})=1$ and $\gcd(\bar{Y},\bar{Z})=1$. From Theorem \ref{teo23} we know that if the curves are projectively equivalent, then 
\begin{equation}\label{eq48}
\kappa_1(\bm{p})(t_0,t_1)-\kappa_1(\bm{q})(u,v)=0,\quad \kappa_2(\bm{p})(t_0,t_1)-\kappa_2(\bm{q})(u,v)=0
\end{equation}
{must hold for two functions $u:=at_0+bt_1,v:=ct_0+dt_1$.} Clearing the denominators of these equations, {where we see $u,v$ as independent variables from $t_0,t_1$,} we define two homogeneous polynomials $E_1$ and $E_2$ in $t_0,t_1,u,v$,
\begin{align}
E_1(t_0,t_1,u,v)&:=U(t_0,t_1)\bar{V}(u,v)-V(t_0,t_1)\bar{U}(u,v) \label{eq49} \\
E_2(t_0,t_1,u,v)&:=Y(t_0,t_1)\bar{Z}(u,v)-Z(t_0,t_1)\bar{Y}(u,v). \label{eq50}
\end{align}

We are interested in the common factors of $E_1$ and $E_2$. Thus, let us write
\begin{equation}\label{eq51}
G(t_0,t_1,u,v):=\gcd(E_1(t_0,t_1,u,v),E_2(t_0,t_1,u,v)).
\end{equation}
Finally, for an arbitrary M\"obius transformation $\varphi(t_0,t_1)=(at_0+bt_1,ct_0+dt_1)=(u,v)$, $ad-bc\neq 0$, we say that
\begin{equation}\label{eq52}
F(t_0,t_1,u,v)=u(ct_0+dt_1)-v(at_0+bt_1)
\end{equation}
is the associated \emph{M\"obius-like factor}. Notice that the condition $ad-bc\neq 0$ guarantees that $F$ is irreducible. 

Then we have the following result. 

\begin{theorem}\label{teo24}
Let $\bm{C}_1,\bm{C}_2$ be two rational algebraic curves properly parametrized by $\bm{p},\bm{q}$ satisfying hypotheses (i-iv), and let $G$ be as in Eq. \eqref{eq51}. If $\bm{C}_1$ and $\bm{C}_2$ are projectively equivalent then there exists a M\"obius-like factor $F$ such that $F$ divides $G$.
\end{theorem}

\begin{color}{black}
\begin{proof} From Theorem \ref{teo23}, the zeroset of $F$ is included in the zerosets of the expressions in Eq. \eqref{eq48}, and therefore in the zerosets of $E_1,E_2$. Thus, the zeroset of $F$ is included in the zeroset of $G$, and therefore, by Study's Lemma (see Section 6.13 in \cite{Fisher}), $F$ divides $G$.
\end{proof}
\end{color}

Thus, in order to compute the M\"obius transformation $\varphi$, we just need to compute the polynomial $G(t_0,t_1,u,v)$ in Eq. \eqref{eq51}, factor it, and look for the M\"obius-like factors. In general we need to factor over the reals, which can be efficiently done with the command {\tt AFactors} in \citet{maple}. Once the $\varphi$ are found, we check whether or not $D(\bfq\circ \varphi)(D(\bfp))^{-1}$ is constant: in the affirmative case, $M=D(\bfq\circ \varphi)(D(\bfp))^{-1}$ defines a projectivity between the curves. For this last part, it is computationally cheaper to compute $D((\bfq\circ \varphi)(t_0))(D(\bfp(t_0)))^{-1}$ for some $t_0\in {\Bbb R}$, and then check whether or not $M\cdot\bfp=\bfq\circ \varphi$ holds. 

Therefore, we get the following algorithm, {\tt Prj3D}, to check whether or not two given rational curves are projectively equivalent. In order to execute the algorithm, we need that not both $\kappa_1,\kappa_2$ are constant. We conjecture that the space curves with both $\kappa_i$ constant may be related to $W$-curves \citep{Sasaki1937}, but at this point we must leave this case out of our study. 

\begin{algorithm}[H]
\caption*{\textbf{Algorithm} $Prj3D$}
\textbf{Input:} \textit{Two proper parametrizations $\bm{p}$ and $\bm{q}$ in homogeneous coordinates such that not both projective curvatures $\kappa_1$, $\kappa_2$ are constant} \\
\textbf{Output:} \textit{Either the list of M\"obius transformations and projectivities, or the warning: "The curves are not projectively equivalent"}

\begin{algorithmic}[1]
    \Procedure{$Prj3D$}{$\bm{p},\bm{q}$}
    
    \State {Compute the set of factors $FS$ of the polynomial $G$ that is defined at \eqref{eq51}.} \Comment{Here we use the {\tt AFactor} function}
    
    \State Check $FS$ to find the set $MF$ of M\"obius-like factors
    
    \If{$MF=\emptyset$}
    	\Return "The curves are not projectively equivalent."
    	\Else 
    	\State Compute the set $MS$ of M\"obius transformations corresponding to $MF$
			\For{$\varphi\in MF$}
			\State Check if $D(\bm{q}(\varphi))D(\bm{p})^{-1}$ is a constant matrix $M$.
			\State In the affirmative case, {\bf return} the projectivity defined by $M$, and the corresponding $\varphi$.
    	\EndFor
    \EndIf
   \EndProcedure
\end{algorithmic}
\end{algorithm}

\begin{color}{black}
\begin{remark}\label{rem-multiples}
Notice that M\"obius-like factors are computed up to a constant, which is coherent with the fact that $\varphi$, which is a projective transformation, is also defined up to a constant. Because of this, the matrix $M$ derived by the algorithm {\tt Prj3D} is also defined up to a constant, i.e. if $F$ is a M\"obius-like factor, which in turn provides a matrix $M$, by picking $\lambda F$ with $\lambda\neq 1$ instead, in general we obtain a multiple $\mu M$ of $M$. However, both $M,\mu M$ define the same projectivity. 
\end{remark}
\end{color}

Below we provide a detailed example to illustrate the steps of the method.

\begin{example}
Consider the curves given by the rational parametrizations
\small
\begin{equation*}
\bm{p}(t_0,t_1)=\begin{pmatrix}
(t_0 - t_1)^4 + 16t_0^4 - 8t_0^3(t_0 - t_1) + 4t_0^2(t_0 - t_1)^2 \\
4t_0^2(t_0 - t_1)^2 \\
8t_0^3(t_0 - t_1) \\
2t_0(t_0 - t_1)((t_0 - t_1)^2 + 4t_0^2)
\end{pmatrix},\quad \bm{q}(t_0,t_1)=\begin{pmatrix}
(t_0 - t_1)^4 + 16t_0^4 \\
2t_0(t_0 - t_1)((t_0 - t_1)^2 + 4t_0^2) \\
2(t_0 - t_1)^3t_0 \\
4t_0^2(t_0 - t_1)^2 
\end{pmatrix}.
\end{equation*}
\normalsize

\noindent The projective curvatures are, in this case, 
\small
\begin{align*}
\kappa_1(\bm{p})(t_0,t_1)&=\frac{\left(17 t_0^{4}-4 t_0^{3} {t_1}+6 t_0^{2} t_1^{2}-4 {t_0} \,t_1^{3}+t_1^{4}\right)^{2}}{384 t_0^{4} \left({t_0}-{t_1}\right)^{2} \left(t_0^{2}-2 {t_0} {t_1}+t_1^{2}\right)} \\
\kappa_2(\bm{p})(t_0,t_1)&=\frac{273 t_0^{8}-72 t_0^{7} {t_1}+124 t_0^{6} t_1^{2}-120 t_0^{5} t_1^{3}+86 t_0^{4} t_1^{4}-56 t_0^{3} t_1^{5}+28 t_0^{2} t_1^{6}-8 {t_0} \,t_1^{7}+t_1^{8}}{96 \left(t_0^{2}-2 {t_0} {t_1}+t_1^{2}\right)^{2} t_0^{4}} 
\end{align*}
\normalsize

\noindent Thus we get
\begin{align*}
E_1&=384 \left(17 t_0^{4}-4 t_0^{3} {t_1}+6 t_0^{2} t_1^{2}-4 {t_0} \,t_1^{3}+t_1^{4}\right)^{2} u^{4} \left(-v+u\right)^{2} \left(u^{2}-2 u v+v^{2}\right)\\
&-384 t_0^{4} \left({t_0}-{t_1}\right)^{2} \left(t_0^{2}-2 {t_0} {t_1}+t_1^{2}\right) \left(17 u^{4}-4 u^{3} v+6 u^{2} v^{2}-4 u \,v^{3}+v^{4}\right)^{2} \\
E_2&=96 \left(273 t_0^{8}-72 t_0^{7} {t_1}+124 t_0^{6} t_1^{2}-120 t_0^{5} t_1^{3}+86 t_0^{4} t_1^{4}-56 t_0^{3} t_1^{5}+28 t_0^{2} t_1^{6}-8 {t_0} \,t_1^{7}+t_1^{8}\right)\\
& \left(u^{2}-2 u v+v^{2}\right)^{2} u^{4}-96 \left(t_0^{2}-2 {t_0} {t_1}+t_1^{2}\right)^{2} t_0^{4}\\
& \left(273 u^{8}-72 u^{7} v+124 u^{6} v^{2}-120 u^{5} v^{3}+86 u^{4} v^{4}-56 u^{3} v^{5}+28 u^{2} v^{6}-8 u \,v^{7}+v^{8}\right)
\end{align*}

\noindent The computation of $G=\gcd(E_1,E_2)$ yields
\begin{align*}
G(t_0,t_1,u,v)&=\left({t_0} v-u {t_1}\right) \left(3 {t_0} u+{t_0} v+u {t_1}-{t_1} v\right) \left(5 {t_0} u-{t_0} v-u {t_1}+{t_1} v\right) \left(2 {t_0} u-{t_0} v-u {t_1}\right) \\
&\left(2 t_0^{2} u^{2}-2 t_0^{2} u v+t_0^{2} v^{2}-2 u^{2} {t_0} {t_1}+u^{2} t_1^{2}\right) \\
&\left(17 t_0^{2} u^{2}-2 t_0^{2} u v+t_0^{2} v^{2}-2 u^{2} {t_0} {t_1}+4 {t_0} {t_1} u v-2 {t_0} {t_1} \,v^{2}+u^{2} t_1^{2}-2 t_1^{2} u v+t_1^{2} v^{2}\right)
\end{align*}

\noindent Factoring $G$, we get the following M\"obius-like factors:

\begin{align*}
f_1&={t_0} u-\frac{1}{2} {t_0} v-\frac{1}{2} {t_1} u \\
f_2&={t_0} u+\frac{1}{3} {t_0} v+\frac{1}{3} {t_1} u-\frac{1}{3} {t_1} v \\
f_3&={t_0} v-{t_1} u \\
f_4&={t_0} u-\frac{1}{5} {t_0} v-\frac{1}{5} {t_1} u+\frac{1}{5} {t_1} v,
\end{align*}
which correspond to the following four M\"obius transformations
\begin{align*}
\varphi_1(t_0,t_1)&=(t_0,2t_0-t_1),&
\varphi_2(t_0,t_1)&=(-t_0+t_1,3t_0+t_1), \\
\varphi_3(t_0,t_1)&=(t_0,t_1),&
\varphi_4(t_0,t_1)&=(t_0-t_1,5t_0-t_1).
\end{align*}

For $i\in\{1,2,3,4\}$, the product $D(q(\varphi_i))D(p)^{-1}$ yields a constant matrix $M_i$, so we get four projectivities $f(\bfx)=M_i\cdot\bfx$ between the curves defined by $\bm{p}$ and $\bm{q}$ corresponding to
\begin{align*}
M_1=\begin{pmatrix}
1 & -1 & 1 & 0 \\
0 & 0 & 0 & -1 \\
0 & 0 & 1 & -1 \\
0 & 1 & 0 & 0
\end{pmatrix}, &\quad M_2=\begin{pmatrix}
16 & -16 & 16 & 0 \\
0 & 0 & 0 & 16 \\
0 & 0 & 16 & 0 \\
0 & 16 & 0 & 0
\end{pmatrix} \\
M_3=\begin{pmatrix}
1 & -1 & 1 & 0 \\
0 & 0 & 0 & 1 \\
0 & 0 & -1 & 1 \\
0 & 1 & 0 & 0
\end{pmatrix}, &\quad M_4=\begin{pmatrix}
16 & -16 & 16 & 0 \\
0 & 0 & 0 & -16 \\
0 & 0 & -16 & 0 \\
0 & 16 & 0 & 0
\end{pmatrix}.
\end{align*}
\end{example}

\section{Implementation and Performance.}\label{Imp}

The algorithm \texttt{Prj3D} was implemented in the computer algebra system \citet{maple}, and was tested on a PC with a $3.6$ GHz Intel Core i$7$ processor and $32$ GB RAM. In order to factor the gcd we use \citet{maple}'s \texttt{AFactors} function, since in general we want to factor over the reals. We want to explicitly mention that the \citet{maple} command {\tt AFactors} works very well in practice. In fact, in our experimentation we observed that most of the time is spent computing the $\gcd$ of the polynomials $E_1$ and $E_2$. Technical details, examples and source codes of the procedures are provided in the first author's personal website \citep{website}.

In this section, first we provide tables and examples to compare the performance of our algorithm with the algorithms in \citep{Hauer201868,BIZZARRI2020112438}. Then we provide a more detailed analysis of our own implementation. 

We recall that the bitsize $\tau$ of an integer $k$ is the integer $\tau=\lceil log_2k\rceil+1$. If the bitsize of an integer is $\tau$, then the number of digits of the integer could be calculated by the formula $d=\lceil log\tau\rceil +1$, where $d$ is the number of digits. 
By an abuse of notation, in this section we have used $\tau$ for representing the maximum bitsize of the coefficients of the components of the parametrization corresponding to a curve.

\subsection{Comparison of the Results.}\label{subsec6.1}

To the best of our knowledge, there are two simple and efficient algorithms to detect the projective equivalences of $3D$ rational curves \citep{Hauer201868,BIZZARRI2020112438}. Although their methods differ, in both cases the authors rely on Gr\"obner bases to solve a polynomial system on the coefficients of the M\"obius transformations corresponding to the equivalences. Thus, in both methods most of the time is spent computing the Gr\"obner basis of the system, which is considerably large. In contrast, our method does not require to solve any polynomial system. Instead, our algorithm computes the M\"obius-like factors by factoring a polynomial of small degree, compared to the degrees in the polynomials involved in the methods \citep{Hauer201868,BIZZARRI2020112438}. The reason is that the polynomial that we need to factor is a gcd of two polynomials where the projective curvatures $\kappa_1$ and $\kappa_2$ are involved. 

In order to compare our results with those in \citep{Hauer201868,BIZZARRI2020112438}, we provide two tables, Table \ref{table3} and Table \ref{table4}, with the timing $t_h$ corresponding to the so-called ``reduced method" in \citep{Hauer201868}, the timing $t_b$ corresponding to \citet{BIZZARRI2020112438}, and the timing $t_{\mbox{our}}$ corresponding to our algorithm. We consider both projective equivalences and symmetries. Since \citet{BIZZARRI2020112438} provide no implementation or tests in their paper, we implemented this algorithm in \citet{maple} to compare with our own, and the timings $t_b$ we are including are the timings obtained with this implementation. For \citep{Hauer201868} we just reproduce the timings in their paper, taking into accout that the machine in \citep{Hauer201868} is similar to ours. We understand that the comparison is unfair because \citep{Hauer201868} uses {\tt Singular} to compute Gr\"obner bases, but perhaps this same fact, i.e. not using the power of {\tt Singular}, that we do not need because we do not compute any Gr\"obner basis, may partially compensate for this unfairness. The results in Table \ref{table3} and Table \ref{table4} show that as the degree of the parametrizations grow, the timings for our algorithm grow much less that the timings for \citep{Hauer201868,BIZZARRI2020112438}, in accordance with the fact that Gr\"obner bases have an exponential complexity. 

Let us present the results corresponding to Table \ref{table3}. The parametrizations used in this table are given in Table \ref{table2}; the first three are taken from \citep{Hauer201868}. Here we have highlighted in blue the best timing for each example. One may notice that our method always beats \citet{BIZZARRI2020112438}, while \citep{Hauer201868} is better for the first two examples, of small degree.

\begin{table}[H]
\centering
\begin{tabular}{l l}
Degree & Parametrization \\
\hline \\
$4$ & $\scalemath{0.8}{\left(\begin{array}{c}
t_0^{4}+t_1^{4} 
\\
 t_0^{3} t_1 +t_0 \,t_1^{3} 
\\
 t_0 \,t_1^{3} 
\\
 t_0^{2} t_1^{2} 
\end{array}\right)}$\\ \\
$6$ & $\scalemath{0.8}{\left(\begin{array}{c}
125 t_0^{6}+450 t_0^{5} t_1 +690 t_0^{4} t_1^{2}+576 t_0^{3} t_1^{3}+276 t_0^{2} t_1^{4}+72 t_0 \,t_1^{5}+8 t_1^{6} 
\\
 -27 t_0^{6}-54 t_0^{5} t_1 -36 t_0^{4} t_1^{2}-8 t_0^{3} t_1^{3} 
\\
 64 t_0^{6}+288 t_0^{5} t_1 +528 t_0^{4} t_1^{2}+504 t_0^{3} t_1^{3}+264 t_0^{2} t_1^{4}+72 {t_0} \,t_1^{5}+8 t_1^{6} 
\\
 21 t_0^{6}+122 t_0^{5} {t_1} +216 t_0^{4} t_1^{2}+168 t_0^{3} t_1^{3}+60 t_0^{2} t_1^{4}+8 {t_0} \,t_1^{5} 
\end{array}\right)}$ \\ \\
$8$ & $\scalemath{0.8}{\left(\begin{array}{c}
625 t_0^{8}+3000 t_0^{7} {t_1} +6400 t_0^{6} t_1^{2}+7920 t_0^{5} t_1^{3}+6216 t_0^{4} t_1^{4}+3168 t_0^{3} t_1^{5}+1024 t_0^{2} t_1^{6}+192 {t_0} \,t_1^{7}+16 t_1^{8} 
\\
 -2027 t_0^{8}-8392 t_0^{7} {t_1} -14344 t_0^{6} t_1^{2}-12768 t_0^{5} t_1^{3}-5960 t_0^{4} t_1^{4}-1056 t_0^{3} t_1^{5}+224 t_0^{2} t_1^{6}+128 {t_0} \,t_1^{7}+16 t_1^{8} 
\\
 1664 t_0^{8}+7744 t_0^{7} {t_1} +16288 t_0^{6} t_1^{2}+20528 t_0^{5} t_1^{3}+17040 t_0^{4} t_1^{4}+9472 t_0^{3} t_1^{5}+3392 t_0^{2} t_1^{6}+704 {t_0} \,t_1^{7}+64 t_1^{8} 
\\
 405 t_0^{8}+1080 t_0^{7} {t_1} +1080 t_0^{6} t_1^{2}+480 t_0^{5} t_1^{3}+80 t_0^{4} t_1^{4} 
\end{array}\right)}$\\ \\
$9$ & $\scalemath{0.8}{\left(\begin{array}{c}
t_0^{9} 
\\
 t_1^{9} 
\\
 t_0^{8} {t_1} +t_0^{6} t_1^{3} 
\\
 t_0^{6} t_1^{3}+t_0^{4} t_1^{5} 
\end{array}\right)}$ \\ \\
$10$ & $\scalemath{0.8}{\left(\begin{array}{c}
49 t_0^{10}-22 t_0^{9} {t_1} +87 t_0^{8} t_1^{2}+84 t_0^{7} t_1^{3}+75 t_0^{6} t_1^{4}-96 t_0^{5} t_1^{5}-28 t_0^{4} t_1^{6}-76 t_0^{3} t_1^{7}-36 t_0^{2} t_1^{8}-55 {t_0} \,t_1^{9}+27 t_1^{10} 
\\
 97 t_0^{10}-97 t_0^{9} {t_1} -73 t_0^{8} t_1^{2}+57 t_0^{7} t_1^{3}+73 t_0^{6} t_1^{4}+64 t_0^{5} t_1^{5}-20 t_0^{4} t_1^{6}+85 t_0^{3} t_1^{7}+99 t_0^{2} t_1^{8}+57 {t_0} \,t_1^{9}+96 t_1^{10} 
\\
 74 t_0^{10}-69 t_0^{9} {t_1} -9 t_0^{8} t_1^{2}+47 t_0^{7} t_1^{3}+44 t_0^{6} t_1^{4}-62 t_0^{5} t_1^{5}+8 t_0^{4} t_1^{6}-84 t_0^{3} t_1^{7}+38 t_0^{2} t_1^{8}-{t_0} \,t_1^{9}+55 t_1^{10} 
\\
-35 t_0^{10}-35 t_0^{9} {t_1} +63 t_0^{8} t_1^{2}+41 t_0^{7} t_1^{3}+16 t_0^{6} t_1^{4}-77 t_0^{5} t_1^{5}+76 t_0^{4} t_1^{6}+95 t_0^{3} t_1^{7}+56 t_0^{2} t_1^{8}-16 {t_0} \,t_1^{9}-95 t_1^{10}\end{array}\right)}$ \\ \\
$11$ & $\scalemath{0.8}{\left(\begin{array}{c}
-62 t_0^{11}-16 t_0^{10} {t_1} +68 t_0^{9} t_1^{2}-15 t_0^{8} t_1^{3}-31 t_0^{7} t_1^{4}+62 t_0^{6} t_1^{5}-14 t_0^{5} t_1^{6}+67 t_0^{4} t_1^{7}+49 t_0^{3} t_1^{8}+52 t_0^{2} t_1^{9}-20 {t_0} \,t_1^{10}-74 t_1^{11} 
\\
 -19 t_0^{11}-68 t_0^{10} {t_1} -48 t_0^{9} t_1^{2}+45 t_0^{8} t_1^{3}+59 t_0^{7} t_1^{4}-96 t_0^{6} t_1^{5}-6 t_0^{5} t_1^{6}+89 t_0^{4} t_1^{7}+41 t_0^{3} t_1^{8}+20 t_0^{2} t_1^{9}+
25 {t_0} \,t_1^{10} 
\\
 -80 t_0^{11}+42 t_0^{10} {t_1} -67 t_0^{9} t_1^{2}+63 t_0^{8} t_1^{3}-81 t_0^{7} t_1^{4}+76 t_0^{6} t_1^{5}-44 t_0^{5} t_1^{6}-59 t_0^{4} t_1^{7}-11 t_0^{3} t_1^{8}-75 t_0^{2} t_1^{9}-84 {t_0} \,t_1^{10}+47 t_1^{11} 
\\
 -27 t_0^{11}-34 t_0^{10} {t_1} +96 t_0^{9} t_1^{2}+82 t_0^{8} t_1^{3}-58 t_0^{7} t_1^{4}+59 t_0^{6} t_1^{5}+36 t_0^{5} t_1^{6}+33 t_0^{4} t_1^{7}+35 t_0^{3} t_1^{8}+27 t_0^{2} t_1^{9}+46 {t_0} \,t_1^{10}+19 t_1^{11} 
\end{array}\right)}$ \\ \\
$12$ & $\scalemath{0.8}{\left(\begin{array}{c}
-62 t_0^{12}-26 t_0^{11} {t_1} +46 t_0^{10} t_1^{2}+65 t_0^{9} t_1^{3}-51 t_0^{8} t_1^{4}+60 t_0^{7} t_1^{5}-56 t_0^{6} t_1^{6}-46 t_0^{5} t_1^{7}+86 t_0^{4} t_1^{8}-31 t_0^{3} t_1^{9}+84 t_0^{2} t_1^{10}+5 {t_0} \,t_1^{11}+25 t_1^{12} 
\\
 -17 t_0^{12}+79 t_0^{11} {t_1} +73 t_0^{10} t_1^{2}-78 t_0^{9} t_1^{3}+13 t_0^{8} t_1^{4}+93 t_0^{7} t_1^{5}+64 t_0^{6} t_1^{6}-70 t_0^{5} t_1^{7}-71 t_0^{4} t_1^{8}-51 t_0^{3} t_1^{9}-71 t_0^{2} t_1^{10}+10 {t_0}t_1^{11} 
\\
 -76 t_0^{12}-25 t_0^{11} {t_1} +38 t_0^{10} t_1^{2}+89 t_0^{9} t_1^{3}-92 t_0^{8} t_1^{4}-84 t_0^{7} t_1^{5}-77 t_0^{6} t_1^{6}-34 t_0^{5} t_1^{7}-20 t_0^{4} t_1^{8}+73 t_0^{3} t_1^{9}-94 t_0^{2} t_1^{10}+99 {t_0}t_1^{11}+18 t_1^{12} 
\\
 39 t_0^{12}-77 t_0^{11} {t_1} -70 t_0^{10} t_1^{2}-49 t_0^{9} t_1^{3}-46 t_0^{8} t_1^{4}+34 t_0^{7} t_1^{5}-84 t_0^{6} t_1^{6}+98 t_0^{5} t_1^{7}+41 t_0^{4} t_1^{8}-46 t_0^{3} t_1^{9}+13 t_0^{2} t_1^{10}-3 {t_0} t_1^{11}+8 t_1^{12} 
\end{array}\right)}$ 
\end{tabular}
\caption{Parametrizations of the curves considered in Section \ref{subsec6.1}}
\label{table2}
\end{table}


\begin{table}[H]
\centering
\begin{tabular}{r r r r r r r r}
\hline
 & \multicolumn{1}{c}{$\sharp$ of} &\multicolumn{1}{c}{$t_b$} & \multicolumn{1}{c}{$t_b$} & \multicolumn{1}{c}{$t_h$}  & \multicolumn{1}{c}{$t_h$} & \multicolumn{1}{c}{$t_{\mbox{our}}$} & \multicolumn{1}{c}{$t_{\mbox{our}}$} \\
  Deg.  & Eqvl. & Symm.  & Eqvl. & Symm. & Eqvl & Symm. & Eqvl\\
\hline
$4$  &  $4$ & $0.344$ & $0.703$ &  \textcolor{blue}{$0.01$}     &  \textcolor{blue}{$0.01$}  &  $0.078$  & $0.219$    \\ 
$6$  &  $4$  & $1.391$ & $2.547$ & \textcolor{blue}{$0.06$}     &  \textcolor{blue}{$0.02$}  &  $0.078$  & $0.172$     \\
$8$  &  $2$  & $3.094$ & $ 2.500$ & $37$       &  $0.78$  &  \textcolor{blue}{$0.063$}  & \textcolor{blue}{$0.188$}      \\
$9$  &  $2$  &   $1.140$  & $1.000$ &        &          &  \textcolor{blue}{$0.016$}  & \textcolor{blue}{$0.031$}       \\
$10$ &  $1$  &  $14.750$  & $10.000$  &          &  		  &	 \textcolor{blue}{$0.422$}  &	\textcolor{blue}{$0.375$}        \\
$11$ &  $1$  &  $31.625$ & $21.172$ &   &  		  &	 \textcolor{blue}{$0.421$}  &	\textcolor{blue}{$0.547$}		    \\
$12$ &  $1$  &  $40.313$  & $41.437$ &         &   		  &	 \textcolor{blue}{$0.625$}  &	\textcolor{blue}{$0.531$}		     \\
\hline
\end{tabular}
\caption{CPU time in seconds for projective symmetries and equivalences for the curves represented by the parametrizations in Table \ref{table2}}
\label{table3}
\end{table}

Now let us introduce Table \ref{table4}. In this table we test random curves with a fixed bitsize $3<\tau<4$ (coefficients are ranges between $-10$ and $10$) as in \citep{Hauer201868}. The first six examples are taken from \citep{Hauer201868}. Again we have highlighted in blue the best timing {among} the methods in \citep{BIZZARRI2020112438}, \citep{Hauer201868} and ours. Our method is only beaten in the first example, of degree 4. For higher degrees not only our algorithm is better, but the growing of the timings is much slower. 

\begin{table}[H]
\centering
\begin{tabular}{r r r r r r r}
\hline
 &  \multicolumn{1}{c}{$t_b$}   &   \multicolumn{1}{c}{$t_b$}    & \multicolumn{1}{c}{$t_h$}  & \multicolumn{1}{c}{$t_h$} & \multicolumn{1}{c}{$t_{\mbox{our}}$} & \multicolumn{1}{c}{$t_{\mbox{our}}$} \\
 Deg.     & Symm. & Eqvl.  & Symm. & Eqvl. & Symm. & Eqvl.\\
\hline
$4$    & $0.390$ & \textcolor{blue}{$0.400$} & \textcolor{blue}{$0.04$}     &  \textcolor{blue}{$0.4$}  &  $0.687$  & $0.860$    \\ 
$5$    & $0.110$ &$0.172$ &  $1$        &  $1.6$  &  \textcolor{blue}{$0.015$}  & \textcolor{blue}{$0.016$}     \\
$6$    &$0.234$ &$0.359$ &  $8.4$      &  $1.2$  &  \textcolor{blue}{$0.047$}  & \textcolor{blue}{$0.031$}      \\
$7$    &$0.610$ &$1.047$ &  $37$       &  $8.6$  &  \textcolor{blue}{$0.187$}  & \textcolor{blue}{$0.063$}       \\
$8$    &$1.579$ &$2.546$ &  $150$      &  $310$  &  \textcolor{blue}{$0.125$}  & \textcolor{blue}{$0.110$}        \\
$9$    &$4.844$ &$4.969$ &  $670$      &  $1700$ &  \textcolor{blue}{$0.297$}  & \textcolor{blue}{$0.343$}		    \\
$10$   &  $10.439$   &$10.484$ &        &         &  \textcolor{blue}{$0.496$}  & \textcolor{blue}{$0.391$}		     \\
$11$   & $22.265$     &$22.438$ &       &     	   &  \textcolor{blue}{$0.625$}  & \textcolor{blue}{$0.453$}		     \\
$12$   &   $42.625$   &$42.797$ &       &   	   &  \textcolor{blue}{$0.906$}  & \textcolor{blue}{$0.547$}		     \\
\hline
\end{tabular}
\caption{CPU time in seconds for projective equivalences and symmetries of random curves with fixed bitsize ($3<\tau<4$)}
\label{table4}
\end{table}

\subsection{Further Tests.}

The tables given in this subsection are provided to better understand the performance of our method and to assist performance testing of similar studies in the future. These tables list timings for homogeneous curve parametrizations with various degrees $m$ and coefficients with bitsizes at most $\tau$. 

\subsubsection{Projective Equivalences and Symmetries of Random Curves.}


In order to generate projectively equivalent curves, we apply the following non-singular matrix and M\"obius transformation to a random parametrization $\bm{q}$ of degree $n$ and bitsize $\tau$. 

\begin{equation*}
M=\begin{pmatrix}
1 & -1 & 1 & 0 \\
0 & 0 & 0 & -1 \\
0 & 0 & -1 & 0 \\
0 & 1 & 0 & 0
\end{pmatrix},\quad \varphi(t_0,t_1)=(-t_0+t_1,2t_0).
\end{equation*}

\noindent Thus, taking $\bm{p}=M\bm{q}(\varphi)$, we run $\texttt{Prj3D}(\bm{p},\bm{q})$ to get the results for projective equivalences, shown in Table \ref{table5}, and $\texttt{Prj3D}(\bm{q},\bm{q})$ for the results in Table \ref{table6} (symmetries); since $\bm{q}$ is randomly generated, in general only the trivial symmetry is expected. Looking at Table \ref{table5} and Table \ref{table6} one observes a smooth increase in the timings for $n\geq 5$; however $n=4$ has, comparatively, higher timings because for degree four curves the homogeneous polynomials $E_1$ and $E_2$ have more redundant common factors {compared to} higher degrees. 

\begin{table}[H]
\centering
\begin{tabular}{r r r r r r r r}
\hline
$t$  & $\tau=4$ & $\tau=8$ & $\tau=16$ & $\tau=32$ & $\tau=64$ & $\tau=128$ & $\tau=256$ \\
\hline
$4$  & $0.703$  & $0.641$  &  $1.500$  & $3.140$   & $4.640$   & $10.281$   & $85.989$  \\ 
$6$  & $0.062$  & $0.062$  &  $0.047$  & $0.063$   & $0.110$   & $0.203$    & $0.531$   \\
$8$  & $0.109$  & $0.125$  &  $0.140$  & $0.172$   & $0.969$   & $1.469$    & $3.578$   \\
$10$ & $0.343$  & $0.531$  &  $0.250$  & $0.344$   & $1.000$   & $2.203$    & $6.000$   \\
$12$ & $0.641$  & $0.718$  &  $0.891$  & $0.860$   & $2.063$   & $3.078$    & $10.719$   \\
$14$ & $0.890$  & $1.188$  &  $1.313$  & $1.641$   & $2.922$   & $5.719$	& $15.704$   \\
$16$ & $1.218$  & $1.172$  &  $1.593$  & $1.875$   & $3.437$   & $7.484$ 	& $23.828$   \\
$18$ & $1.797$  & $1.844$  &  $2.313$  & $2.688$   & $5.656$   & $9.890$	& $32.985$   \\
$20$ & $2.344$  & $2.125$  &  $3.281$  & $4.219$   & $7.297$   & $14.203$	& $46.282$  \\
$22$ & $2.985$  & $3.609$  &  $4.203$  & $5.391$   & $8.781$  & $18.062$	& $65.000$  \\
$24$ & $4.125$ & $4.672$  &  $4.859$  & $6.344$   & $11.110$  & $20.954$	& $74.766$  \\
\hline
\end{tabular}
\caption{CPU times in seconds for projective equivalences of
random curves with various degrees $m$ and bitsizes at most $\tau$}
\label{table5}
\end{table}

\begin{table}[H]
\centering
\begin{tabular}{r r r r r r r r}
\hline
$t$  & $\tau=4$ & $\tau=8$ & $\tau=16$ & $\tau=32$ & $\tau=64$ & $\tau=128$ & $\tau=256$ \\
\hline
$4$  & $0.688$  & $1.438$  &  $0.797$  & $2.078$   & $3.125$   & $9.891$   & $219.750$  \\ 
$6$  & $0.110$  & $0.016$  &  $0.047$  & $0.062$   & $0.093$   & $0.172$    & $0.531$   \\
$8$  & $0.110$  & $0.109$  &  $0.438$  & $0.625$   & $0.250$   & $0.593$    & $2.344$   \\
$10$ & $0.344$  & $0.235$  &  $0.562$  & $0.781$   & $1.047$   & $2.297$    & $5.203$   \\
$12$ & $0.547$  & $0.750$  &  $0.812$  & $1.609$   & $2.281$   & $3.547$    & $9.281$   \\
$14$ & $0.688$  & $0.922$  &  $1.672$  & $1.546$   & $3.297$   & $5.172$	& $17.531$   \\
$16$ & $1.297$  & $1.609$  &  $1.828$  & $2.672$   & $5.219$   & $7.360$ 	& $22.156$   \\
$18$ & $2.047$  & $1.750$  &  $2.156$  & $3.281$   & $6.797$   & $10.907$	& $34.718$   \\
$20$ & $2.562$  & $2.281$  &  $3.516$  & $4.687$  & $8.656$  & $13.906$	& $45.093$  \\
$22$ & $3.375$  & $3.469$ &  $4.735$  & $5.500$  & $11.609$  & $16.859$	& $57.500$  \\
$24$ & $4.093$ & $4.703$ &  $5.391$ & $7.375$  & $12.781$  & $22.343$	& $75.469$  \\
\hline
\end{tabular}
\caption{CPU times in seconds for projective symmetries (only trivial symmetry) of
random curves with various degrees $m$ and bitsizes at most $\tau$}
\label{table6}
\end{table}

\subsubsection{Projective Symmetries of Random Curves with Central Inversion.}

To analyze the effect of an additional non-trivial symmetry, we considered random parametrizations $\bm{p}(t_0,t_1)=(\bm{p}_0(t_0,t_1),\bm{p}_1(t_0,t_1),\bm{p}_2(t_0,t_1),
\bm{p}_3(t_0,t_1))$ with a symmetric $\bm{p}_0(t_0,t_1)$ and an anti-symmetric triple $\bm{p}_1(t_0,t_1)$, $\bm{p}_2(t_0,t_1)$ and $\bm{p}_3(t_0,t_1)$ of the same even-degree $m$ and with bitsize at most $\tau$, i.e. of the form
\begin{align*}
\bm{p}_0(t_0,t_1)& =c_{0,0}t_0^m+c_{1,0}t_0^{m-1}t_1+...+c_{1,0}t_0t_1^{m-1}+c_{0,0}t_1^m \\
\bm{p}_i(t_0,t_1)& =c_{0,i}t_0^m+c_{1,i}t_0^{m-1}t_1+...-c_{1,i}t_0t_1^{m-1}-c_{0,i}t_1^m ,
\end{align*}
with $c_{\frac{m}{2},i}=0$ for all $i\in\{1,2,3\}$. Since $\bm{p}(t_1,t_0)=(\bm{p}_0(t_0,t_1),-\bm{p}_1(t_0,t_1),-\bm{p}_2(t_0,t_1),-\bm{p}_3(t_0,t_1))$, such homogeneous parametric curves are invariant under a central inversion with respect to the origin.

Table \ref{table7} lists the timings to detect projective symmetries (central inversions, in this case) of random curves with various degrees $m$ and bitsizes at most $\tau$. As expected, one can see that the computation times remain within the same magnitude order with respect to previous tables. 

\begin{table}[H]
\centering
\begin{tabular}{r r r r r r r r}
\hline
$t$  & $\tau=4$ & $\tau=8$ & $\tau=16$ & $\tau=32$ & $\tau=64$ & $\tau=128$ & $\tau=256$ \\
\hline
$8$  & $0.078$  & $0.078$  &  $0.078$  & $0.093$   & $0.141$   & $0.500$    & $1.109$   \\
$10$ & $0.234$  & $0.172$  &  $0.313$  & $0.453$   & $0.781$   & $1.609$    & $4.516$   \\
$12$ & $0.360$  & $0.516$  &  $0.531$  & $0.704$   & $1.282$   & $3.031$    & $8.140$   \\
$14$ & $0.625$  & $0.812$  &  $0.703$  & $1.078$   & $2.062$   & $5.000$	& $13.032$   \\
$16$ & $0.921$  & $1.047$  &  $1.203$  & $1.937$   & $3.125$   & $7.172$ 	& $20.109$   \\
$18$ & $1.329$  & $1.250$  &  $1.578$  & $2.516$   & $4.969$   & $9.141$	& $28.922$   \\
$20$ & $1.765$  & $1.922$  &  $2.282$  & $3.390$   & $5.594$   & $14.000$	& $39.125$  \\
\hline
\end{tabular}
\caption{CPU times in seconds for projective symmetries (central inversion) of
random curves with various degrees $m$ and bitsizes at most $\tau$}
\label{table7}
\end{table}

\subsubsection{Projective Equivalences of Non-equivalent Curves.}

In the last table that we present here, Table \ref{table8}, we generate both curves randomly, so in general no projective equivalence is expected. Table \ref{table8} shows the computation times for non-equivalent random curves with various degrees $m$ and bitsizes at most $\tau$. As expected, the timings are faster than those of Table \ref{table5}, Table \ref{table6}, Table \ref{table7}. The reason is that in most cases the gcd $G$ is constant and therefore the algorithm finishes earlier. 
  
\begin{table}[H]
\centering
\begin{tabular}{r r r r r r r r}
\hline
$t$  & $\tau=4$ & $\tau=8$ & $\tau=16$ & $\tau=32$ & $\tau=64$ & $\tau=128$ & $\tau=256$ \\
\hline
$4$  & $0.016$  & $0.015$  &  $0.016$  & $0.015$   & $0.015$   & $0.015$    & $0.015$   \\
$6$  & $0.094$  & $0.329$  &  $0.031$  & $0.047$   & $0.047$   & $0.046$    & $0.688$   \\
$8$  & $0.062$  & $0.062$  &  $0.078$  & $0.094$   & $0.110$   & $0.829$    & $0.328$   \\
$10$ & $0.313$  & $0.141$  &  $0.157$  & $0.453$   & $0.796$   & $0.328$    & $0.656$   \\
$12$ & $0.281$  & $0.250$  &  $0.718$  & $0.297$   & $0.391$   & $0.937$    & $1.343$   \\
$14$ & $0.547$  & $0.625$  &  $0.391$  & $0.703$   & $0.953$   & $1.344$	& $2.500$   \\
$16$ & $0.922$  & $0.547$  &  $0.969$  & $0.609$   & $1.031$   & $1.937$ 	& $3.595$   \\
$18$ & $1.062$  & $1.046$  &  $1.047$  & $1.313$   & $1.500$   & $2.922$	& $4.266$   \\
$20$ & $1.438$  & $1.375$  &  $1.312$  & $1.890$   & $2.344$   & $3.594$	& $6.359$  \\
$22$ & $2.109$  & $1.704$  &  $1.609$  & $2.187$   & $3.219$   & $4.453$	& $8.891$   \\
$24$ & $1.719$  & $2.343$  &  $2.391$  & $2.782$   & $4.234$   & $6.735$	& $11.078$  \\
\hline
\hline
\end{tabular}
\caption{CPU times in seconds for non-equivalent random curves with various degrees $m$ and bitsizes at most $\tau$}
\label{table8}
\end{table}

\subsubsection{Effect of the Bitsize and Degree on the Algorithm.}

Our implementation can deal with curves of degree $24$ and bitsize $256$ at the same time. When we attempt to solve the problem for higher degrees and bitsizes at the same time, the computer runs out of memory. However, by fixing either the bitsize or the degree we are able to go further and explore the limits of the method. Here we present the results of two different tests on random homogeneous parametrizations, one for a fixed bitsize and one for a fixed degree. In these tests the second parametrization is obtained by applying a projective transformation and a M\"obius transformation to the first, random, parametrization. For the first test we fix the bitsize at $4$, and increase the degree up to $128$; for the second test, we fix the degree at $8$, and increase the bitsize up to to $2^{12}$. The results are shown in Figure \ref{fig1}; Figure \ref{left} exhibits log plots of CPU times against the degree, and Figure \ref{right} exhibits non-log plots of CPU times against the coefficient bitsizes. The data were analysed using the \texttt{PowerFit} function of the \texttt{Statistics} package of \citet{maple}. Thus, as a function of the degree $m$, the CPU time $t$ satisfies 
\begin{equation}\label{eq64}
t \sim \alpha m^\beta, \;\;\; \alpha \approx 2.0*10^{-4}, \;\;\; \beta \approx 3.1,
\end{equation} 
and as a function of the bitsize $\tau$, the CPU time $t$ satisfies
\begin{equation}\label{eq65}
t \sim \alpha \tau^\beta, \;\;\; \alpha \approx 5.7*10^{-2}, \;\;\; \beta \approx 0.6.
\end{equation} 

\begin{figure}[H]
\centering
\begin{subfigure}[b]{0.45\textwidth}\centering
\centering
        \includegraphics[width=\textwidth]{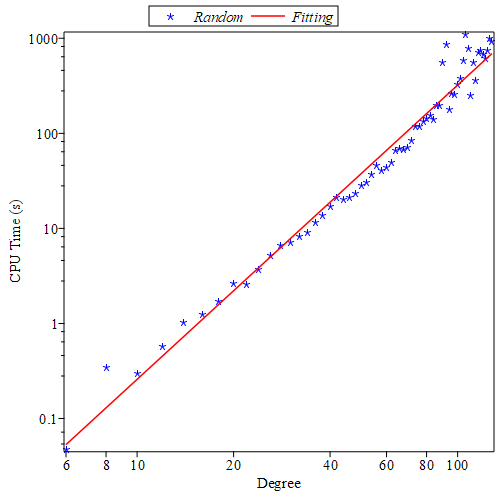}
        \caption{$t$ versus $m$}
        \label{left}
    \end{subfigure}
\begin{subfigure}[b]{0.45\textwidth}\centering
\centering
        \includegraphics[width=\textwidth]{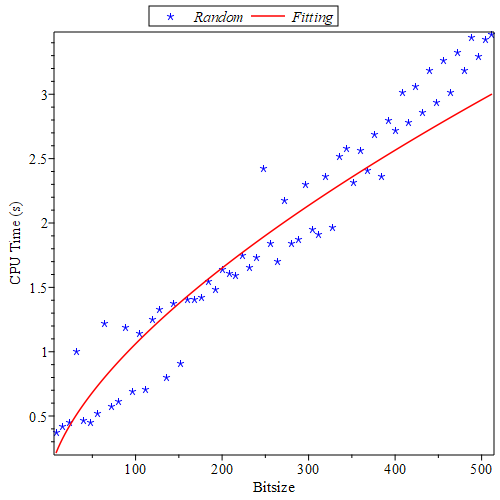}
        \caption{$t$ versus $\tau$}
        \label{right}
    \end{subfigure}
\caption{\ref{left}: CPU time $t$ in seconds versus the degree $m$ with a fixed bitsize $\tau=4$. The asterisk represents the computations corresponding to degrees and line represents the fitting by the power law \eqref{eq64}. \ref{right}: CPU time $t$ in seconds versus the bitsize $\tau$ with a fixed degree $n=8$. The asterisk represents the computations corresponding to bitsizes and line represents the fitting by the power law \eqref{eq65}.} 
\label{fig1}
\end{figure}

\section{Conclusion and Future Work.}\label{sec-conclusion}

We have presented a new approach to the problem of detecting projective equivalences of space rational curves. The method is inspired in the ideas developed in \citep{Alcazar201551} for computing symmetries of 3D space rational curves{, as well as in the theory of differential invariants}. The method proceeds by introducing two {rational functions}, called projective curvatures, so that the projectivities between the curves are derived after computing the M\"obius-like factors of two polynomials built from the projective curvatures. From an algorithmic point of view, it only requires gcd computing and factoring of a polynomial of relatively small degree, and therefore differs from previous approaches, where big polynomial systems were used. The experimentation carried out with \citet{maple} shows that the method is efficient and works better than previous approaches as the degree of the curves involved in the computation grow.  

\begin{color}{black}
We conjecture that the method is generalizable to other dimensions, transformation groups and parametric varieties (e.g. surfaces). The essential requirement is to know what kind of transformation we have in the parameter space (in the case treated in this paper, M\"obius transformations). The sketched general scheme is: 
\begin{itemize}
\item [(1)] {\it Generate invariants} (in the case treated in this paper, the $I_i$). In general, these invariants will not have a nice behavior with respect to the transformations in the parameter space; that is what we mean when we speak about ``commuting with M\"obius transformations"', in our case. 
\item [(2)] From the invariants in (1), {\it generate other invariants that behave nicely with respect to the transformations in the parameter space}. In our case, these are the $\kappa_i$. In general, this is a problem of eliminating variables cleverly.
\item [(3)] From the invariants in (2), {\it find efficiently the transformation in the parameter space}. In our case, this is done by gcd computation and factoring. 
\end{itemize}
This is a general outline that needs to be adapted depending on the dimension, the transformation group and the varieties involved. We do not have at the moment a proof that this scheme always works, although our ongoing investigation suggests that it certainly succeeds for planar rational curves. So the theoretical study of the viability, generality and correctness of the suggested strategy is an open question that requires further investigation. 

For rational curves in dimension $n$, the $I_i$ in step (1) would be
\begin{align*}
I_1(\bm{p}):=\dfrac{\Vert \bm{p}_{t_0^{n+1}}\, \bm{p}_{t_1}\, \bm{p}_{t_0^2}\, \cdots \bm{p}_{t_0^n} \Vert}{\Vert \bm{p}_{t_0}\, \bm{p}_{t_1}\, \bm{p}_{t_0^2}\, \cdots \bm{p}_{t_0^n} \Vert},\, I_2(\bm{p}):=\dfrac{\Vert \bm{p}_{t_0}\, \bm{p}_{t_0^{n+1}}\, \bm{p}_{t_0^2}\, \cdots \bm{p}_{t_0^n} \Vert}{\Vert \bm{p}_{t_0}\, \bm{p}_{t_1}\, \bm{p}_{t_0^2}\, \cdots \bm{p}_{t_0^n} \Vert},\ldots, I_{n+1}(\bm{p}):=\dfrac{\Vert \bm{p}_{t_0}\, \bm{p}_{t_1}\, \bm{p}_{t_0^2}\, \cdots \bm{p}_{t_0^{n+1}} \Vert}{\Vert \bm{p}_{t_0}\, \bm{p}_{t_1}\, \bm{p}_{t_0^2}\, \cdots \bm{p}_{t_0^n} \Vert}. 
\end{align*}
However, deriving the functions in step (2) for a general $n$ is more complicated and requires further research.  
\end{color}

Additionally, the method opens other interesting theoretical questions as well, like the geometric interpretation of the curvatures introduced in this paper, as well as a study of the curves where these curvatures are constant, which is a particular case that the algorithm in this paper cannot deal with. 

\bibliography{mybibfile}

\newpage
\appendix

\section{Proof of Lemma \ref{lem9}.}\label{AppendixA}
\begin{proof} (of Lemma \ref{lem9})
Using the notation in Section \ref{prelmn} for the parametrization $\bm{u}$, we get that
\begin{equation}\label{eq6}
\bm{u}(t_0,t_1)=A\cdot\begin{bmatrix}
t_1^n \\
t_1^{n-1}t_0 \\
\vdots \\
t_1t_0^{n-1} \\
t_0^n
\end{bmatrix},
\end{equation}
where $A$ is the $4\times (n+1)$ coefficient matrix corresponding to $\bm{u}$. Differentiating \eqref{eq6} we have
\begin{align*}
\bm{u}_{t_0}(t_0,t_1) & =A\cdot T_0 , & T_0 &=\begin{bmatrix}
0 \\
t_1^{n-1} \\
\vdots \\
(n-1)t_1t_0^{n-2} \\
nt_0^{n-1}
\end{bmatrix} \\
\bm{u}_{t_1}(t_0,t_1) & =A\cdot T_1, & T_1 &=\begin{bmatrix}
nt_1^{n-1} \\
(n-1)t_1^{n-2}t_0 \\
\vdots \\
t_0^{n-1} \\
0
\end{bmatrix} \\
\bm{u}_{t_0^2}(t_0,t_1) & =A\cdot T_2, & T_2 &=\begin{bmatrix}
0 \\
0 \\
\vdots \\
(n-1)(n-2)t_1t_0^{n-3} \\
n(n-1)t_0^{n-2}
\end{bmatrix} \\
\bm{u}_{t_0^3}(t_0,t_1) & =A\cdot T_3, & T_3 &=\begin{bmatrix}
0 \\
0 \\
0 \\
\vdots \\
(n-1)(n-2)(n-3)t_1t_0^{n-4} \\
n(n-1)(n-2)t_0^{n-3}
\end{bmatrix}.
\end{align*}

Thus

\begin{equation}\label{eq7}
D(\bm{u})=\left[ A\cdot T_0\,\, A\cdot T_1\,\, A\cdot T_2\,\, A\cdot T_3\right] =A\cdot T,
\end{equation}
where $T=\left[ T_0\,\, T_1\,\, T_2\,\, T_3\right]$.  

Now since by hypothesis $\bm{C}$ is not contained in a hyperplane, $\mbox{rank}(A)=4$. Additionally, since $n\geq 4$ we also have $\mbox{rank}(T)=4$. But then the product $A\cdot T$ must also have full rank, and therefore $\Delta(\bm{u})=|A\cdot T|$ cannot be identically zero. 
\end{proof}

\section{Projective curvatures are well defined.}\label{AppendixB}

In this appendix we prove {Lemma \ref{denomnot} in Section \ref{subsec-proj-curv}}. This result follows from a more general result, namely that the $I_i$, $i\in \{1,2,3,4\}$, introduced in Subsection \ref{subsec3.2} are algebraically independent{. This last result will be proven first.}

The following technical results regarding the properties of the {$I_i$} will be later used to prove the algebraic independence of the $I_i$.

\begin{lemma}\label{lem11}
Let $\bm{C}$ be a rational algebraic curve of degree $n$ properly parametrized by $\bm{p}(t_0,t_1)=(\bm{p}_0(t_0,t_1),\bm{p}_1(t_0,t_1),\\
\bm{p}_2(t_0,t_1),\bm{p}_3(t_0,t_1))$ satisfying \textit{hypotheses} \textit{(i-iv)}. Then $t_1$ is a factor of $\Delta(\bm{p})$.
\end{lemma}

\begin{proof}
Since the degree of $\bm{C}$ is $n$, we can write
\begin{equation*}
\bm{p}_k(t_0,t_1)=\sum_{r=0}^{n}a_{r,k}t_0^{n-r}t_1^{r},\, \, \, 0\leq k\leq 3.
\end{equation*}
The partial derivatives of order $i$ of these polynomials with respect to $t_0$ are
\begin{equation*}
\dfrac{\partial^i \bm{p}_k}{\partial t_0^i}(t_0,t_1)=\sum_{r=0}^{n-i}\dfrac{(n-r)!}{(n-r-i)!}a_{r,k}t_0^{n-r-i}t_1^{r},
\end{equation*}
with $i\in\{1,2,3\}$.

Additionally,
\begin{align*}
\Delta(\bm{p})=\lVert \bm{p}_{t_0} \, \bm{p}_{t_1}\, \bm{p}_{t_0^2}\, \bm{p}_{t_0^3}\rVert=\begin{vmatrix}
\dfrac{\partial \bm{p}_0}{\partial t_0}(t_0,t_1) \,\dfrac{\partial \bm{p}_0}{\partial t_1}(t_0,t_1), \dfrac{\partial^2 \bm{p}_0}{\partial t_0^2}(t_0,t_1), \dfrac{\partial^3 \bm{p}_0}{\partial t_0^3}(t_0,t_1)\\
\dfrac{\partial \bm{p}_1}{\partial t_0}(t_0,t_1) \,\dfrac{\partial \bm{p}_1}{\partial t_1}(t_0,t_1), \dfrac{\partial^2 \bm{p}_1}{\partial t_0^2}(t_0,t_1), \dfrac{\partial^3 \bm{p}_1}{\partial t_0^3}(t_0,t_1)\\
\dfrac{\partial \bm{p}_2}{\partial t_0}(t_0,t_1) \,\dfrac{\partial \bm{p}_2}{\partial t_1}(t_0,t_1), \dfrac{\partial^2 \bm{p}_2}{\partial t_0^2}(t_0,t_1), \dfrac{\partial^3 \bm{p}_2}{\partial t_0^3}(t_0,t_1)\\
\dfrac{\partial \bm{p}_3}{\partial t_0}(t_0,t_1) \,\dfrac{\partial \bm{p}_3}{\partial t_1}(t_0,t_1), \dfrac{\partial^2 \bm{p}_3}{\partial t_0^2}(t_0,t_1), \dfrac{\partial^3 \bm{p}_3}{\partial t_0^3}(t_0,t_1)
\end{vmatrix}.
\end{align*}
In order to compute $\Delta(\bm{p})$, we expand the above determinant by the first column,
\begin{align*}
\Delta(\bm{p})=\dfrac{\partial \bm{p}_0}{\partial t_0}(t_0,t_1)\begin{vmatrix}
\dfrac{\partial \bm{p}_1}{\partial t_1}(t_0,t_1), \dfrac{\partial^2 \bm{p}_1}{\partial t_0^2}(t_0,t_1), \dfrac{\partial^3 \bm{p}_1}{\partial t_0^3}(t_0,t_1)\\
\dfrac{\partial \bm{p}_2}{\partial t_1}(t_0,t_1), \dfrac{\partial^2 \bm{p}_2}{\partial t_0^2}(t_0,t_1), \dfrac{\partial^3 \bm{p}_2}{\partial t_0^3}(t_0,t_1)\\
\dfrac{\partial \bm{p}_3}{\partial t_1}(t_0,t_1), \dfrac{\partial^2 \bm{p}_3}{\partial t_0^2}(t_0,t_1), \dfrac{\partial^3 \bm{p}_3}{\partial t_0^3}(t_0,t_1)
\end{vmatrix}-\dfrac{\partial \bm{p}_1}{\partial t_0}(t_0,t_1)\begin{vmatrix}
\dfrac{\partial \bm{p}_0}{\partial t_1}(t_0,t_1), \dfrac{\partial^2 \bm{p}_0}{\partial t_0^2}(t_0,t_1), \dfrac{\partial^3 \bm{p}_0}{\partial t_0^3}(t_0,t_1)\\
\dfrac{\partial \bm{p}_2}{\partial t_1}(t_0,t_1), \dfrac{\partial^2 \bm{p}_2}{\partial t_0^2}(t_0,t_1), \dfrac{\partial^3 \bm{p}_2}{\partial t_0^3}(t_0,t_1)\\
\dfrac{\partial \bm{p}_3}{\partial t_1}(t_0,t_1), \dfrac{\partial^2 \bm{p}_3}{\partial t_0^2}(t_0,t_1), \dfrac{\partial^3 \bm{p}_3}{\partial t_0^3}(t_0,t_1)
\end{vmatrix}
\end{align*}
\begin{align*}
+\dfrac{\partial \bm{p}_2}{\partial t_0}(t_0,t_1) \begin{vmatrix}
\dfrac{\partial \bm{p}_0}{\partial t_1}(t_0,t_1), \dfrac{\partial^2 \bm{p}_0}{\partial t_0^2}(t_0,t_1), \dfrac{\partial^3 \bm{p}_0}{\partial t_0^3}(t_0,t_1)\\
\dfrac{\partial \bm{p}_1}{\partial t_1}(t_0,t_1), \dfrac{\partial^2 \bm{p}_1}{\partial t_0^2}(t_0,t_1), \dfrac{\partial^3 \bm{p}_1}{\partial t_0^3}(t_0,t_1)\\
\dfrac{\partial \bm{p}_3}{\partial t_1}(t_0,t_1), \dfrac{\partial^2 \bm{p}_3}{\partial t_0^2}(t_0,t_1), \dfrac{\partial^3 \bm{p}_3}{\partial t_0^3}(t_0,t_1)
\end{vmatrix}-\dfrac{\partial \bm{p}_3}{\partial t_0}(t_0,t_1) \begin{vmatrix}
\dfrac{\partial \bm{p}_0}{\partial t_1}(t_0,t_1), \dfrac{\partial^2 \bm{p}_0}{\partial t_0^2}(t_0,t_1), \dfrac{\partial^3 \bm{p}_0}{\partial t_0^3}(t_0,t_1)\\
\dfrac{\partial \bm{p}_1}{\partial t_1}(t_0,t_1), \dfrac{\partial^2 \bm{p}_1}{\partial t_0^2}(t_0,t_1), \dfrac{\partial^3 \bm{p}_1}{\partial t_0^3}(t_0,t_1)\\
\dfrac{\partial \bm{p}_2}{\partial t_1}(t_0,t_1), \dfrac{\partial^2 \bm{p}_2}{\partial t_0^2}(t_0,t_1), \dfrac{\partial^3 \bm{p}_2}{\partial t_0^3}(t_0,t_1)
\end{vmatrix}.
\end{align*}

Now let us consider the cofactors of the elements in the first column of each of these four determinants: 
\begin{align*}
& \dfrac{\partial ^2 \bm{p}_k}{\partial t_0^{2}}(t_0,t_1)\dfrac{\partial ^3 \bm{p}_l}{\partial t_0^{3}}(t_0,t_1)-\dfrac{\partial ^3 \bm{p}_l}{\partial t_0^{3}}(t_0,t_1)\dfrac{\partial ^2 \bm{p}_k}{\partial t_0^{2}}(t_0,t_1)=\\
&=\sum_{r=0}^{n-2}\sum_{s=0}^{n-3}\dfrac{(n-r)!}{(n-r-2)!}\dfrac{(n-s)!}{(n-s-3)!}a_{r,k}a_{s,l}t_0^{2n-(r+s+5)}t_1^{r+s}\\
& -\sum_{r=0}^{n-2}\sum_{s=0}^{n-3}\dfrac{(n-r)!}{(n-r-2)!}\dfrac{(n-s)!}{(n-s-3)!}a_{s,k}a_{r,l}t_0^{2n-(r+s+5)}t_1^{r+s}\\
&=\sum_{r=0}^{n-2}\sum_{s=0}^{n-3}\dfrac{(n-r)!}{(n-r-2)!}\dfrac{(n-s)!}{(n-s-3)!}(a_{s,k}a_{r,l}-a_{r,k}a_{s,l})t_0^{2n-(r+s+5)}t_1^{r+s},
\end{align*}
where $ 0\leq k<l \leq 3.$ One can easily see that $a_{s,k}a_{r,l}-a_{r,k}a_{s,l}=0$ whenever $r=s$. Thus, whenever $r+s>0$ the factor $t_1$ is present in each cofactor, and therefore $t_1$ is a factor of $\Delta(\bm{p})$.
\end{proof}

The following lemma follows directly from properties of determinants. Here we recall the definition of the {rational expressions} $I_i$ introduced in Subsection \ref{subsec3.2},

\begin{equation}\label{eq8}
I_1(\bm{p}):=\dfrac{A_1(\bm{p})}{\Delta(\bm{p})},\, I_2(\bm{p}):=\dfrac{A_2(\bm{p})}{\Delta(\bm{p})},\, I_3(\bm{p}):=\dfrac{A_3(\bm{p})}{\Delta(\bm{p})},\, I_4(\bm{p}):=\dfrac{A_4(\bm{p})}{\Delta(\bm{p})}.
\end{equation}

\begin{lemma}\label{lem12}
Let $\bm{C}$ be a rational algebraic curve properly parametrized by $\bm{p}$ satisfying \textit{hypotheses} \textit{(i-iv)}. Then
\begin{align*}
A_1(\bm{p})_{t_0} &= A_{5,1}(\bm{p})-\dfrac{n-1}{t_1}A_2(\bm{p}) \\
A_2(\bm{p})_{t_0} &= A_{5,2}(\bm{p}) \\
A_3(\bm{p})_{t_0} &= A_{5,3}(\bm{p})+\dfrac{t_0}{t_1}A_2(\bm{p})-A_1(\bm{p})\\
A_4(\bm{p})_{t_0} &= A_{5,4}(\bm{p})-A_3(\bm{p}),
\end{align*}
where 
\begin{align*}
A_{5,1}(\bm{p})&=\lVert \bm{p}_{t_0^5}\, \bm{p}_{t_1}\, \bm{p}_{t_0^2}\, \bm{p}_{t_0^3}\rVert, &A_{5,2}(\bm{p})=\lVert\bm{p}_{t_0}\, \bm{p}_{t_0^5}\, \bm{p}_{t_0^2}\, \bm{p}_{t_0^3}\rVert,\\
A_{5,3}(\bm{p})&=\lVert \bm{p}_{t_0}\, \bm{p}_{t_1}\, \bm{p}_{t_0^5}\, \bm{p}_{t_0^3}\rVert,&A_{5,4}(\bm{p})=\lVert \bm{p}_{t_0}\, \bm{p}_{t_1}\, \bm{p}_{t_0^2}\, \bm{p}_{t_0^5}\rVert.
\end{align*}
\end{lemma}

The next lemma is the standard bracket syzygy in the classical invariant theory \citep{olver_1999}.

\begin{lemma}\label{lem13}
Let $\bm{x}_0,\bm{x}_1,...,\bm{x}_n,\bm{y}_2,\bm{y}_3,...,\bm{y}_n\in\mathbb{E}^n$. Then 
\begin{equation}\label{eq9}
\lVert \bm{x}_1\,\bm{x}_2\, ...\, \bm{x}_n \rVert\lVert \bm{x}_0\,\bm{y}_2\, ...\, \bm{y}_n \rVert-\lVert \bm{x}_0\,\bm{x}_2\, ...\, \bm{x}_n \rVert\lVert \bm{x}_1\,\bm{y}_2\, ...\, \bm{y}_n \rVert-...-\lVert \bm{x}_1\, ...\,\bm{x}_{n-1}\, \bm{x}_0 \rVert\lVert \bm{x}_n\,\bm{y}_2\, ...\, \bm{y}_n \rVert=0.
\end{equation}
\end{lemma}

The next result introduces new relationships between the numerators of some of the invariants $I_i$, and the polynomials introduced in the statement of Lemma \ref{lem12}.

\begin{lemma}\label{lem14}
Let $\bm{C}$ be a rational algebraic curve properly parametrized by $\bm{p}$ satisfying \textit{hypotheses} \textit{(i-iv)}. Then we get that
\begin{align*}
\Delta(\bm{p}_{t_0}) &=-\dfrac{n-1}{t_1} A_{2}(\bm{p}) \\
A_1(\bm{p}_{t_0}) &=\dfrac{n-1}{t_1}(I_3(\bm{p})A_{5,2}(\bm{p})-I_2(\bm{p})A_{5,3}(\bm{p}))-\dfrac{t_0}{t_1}(I_2(\bm{p})A_{5,1}(\bm{p})-I_1(\bm{p})A_{5,2}(\bm{p})) \\
A_2(\bm{p}_{t_0}) &=I_2(\bm{p})A_{5,1}(\bm{p})-I_1(\bm{p})A_{5,2}(\bm{p}) \\
A_3(\bm{p}_{t_0}) &=\dfrac{n-1}{t_1}(I_4(\bm{p})A_{5,2}(\bm{p})-I_2(\bm{p})A_{5,4}(\bm{p})).
\end{align*}
\end{lemma}

\begin{proof}
We prove the lemma only for $A_1(\bm{p}_{t_0})$; the proofs for the other equalities are similar. By definition, we have $A_1(\bm{p}_{t_0})=\lVert\bm{p}_{t_0^5}\, \bm{p}_{t_1t_0}\, \bm{p}_{t_0^3}\, \bm{p}_{t_0^4}\rVert$. Using Euler's Homogeneous Function Theorem, we get that
\begin{equation*}
A_1(\bm{p}_{t_0})=\dfrac{n-1}{t_1}\lVert\bm{p}_{t_0^5}\, \bm{p}_{t_0}\, \bm{p}_{t_0^3}\, \bm{p}_{t_0^4}\rVert-\dfrac{t_0}{t_1}\lVert\bm{p}_{t_0^5}\, \bm{p}_{t_0^2}\, \bm{p}_{t_0^3}\, \bm{p}_{t_0^4}\rVert.
\end{equation*}

Let us apply Lemma \ref{lem13} to the vectors $\bm{p}_{t_1},\bm{p}_{t_0^5},\bm{p}_{t_0},\bm{p}_{t_0^3},\bm{p}_{t_0^4}$ and $\bm{p}_{t_0},\bm{p}_{t_0^2},\bm{p}_{t_0^3}$. Eliminating the zero determinants, we have
\begin{align*}
\lVert\bm{p}_{t_0^5}\, \bm{p}_{t_0}\, \bm{p}_{t_0^3}\, \bm{p}_{t_0^4}\rVert\lVert\bm{p}_{t_1}\, \bm{p}_{t_0}\, \bm{p}_{t_0^2}\, \bm{p}_{t_0^3}\rVert-\lVert\bm{p}_{t_1}\, \bm{p}_{t_0}\, \bm{p}_{t_0^3}\, \bm{p}_{t_0^4}\rVert\lVert\bm{p}_{t_0^5}\, \bm{p}_{t_0}\, \bm{p}_{t_0^2}\, \bm{p}_{t_0^3}\rVert-\lVert\bm{p}_{t_0^5}\, \bm{p}_{t_0}\, \bm{p}_{t_0^3}\, \bm{p}_{t_1}\rVert\lVert\bm{p}_{t_0^4}\, \bm{p}_{t_0}\, \bm{p}_{t_0^2}\, \bm{p}_{t_0^3}\rVert=0.
\end{align*}

By the definitions of $A_2(\bm{p}),A_3(\bm{p}),A_{5,2}(\bm{p}),A_{5,3}(\bm{p}),\Delta(\bm{p})$, we obtain

\begin{align*}
\lVert\bm{p}_{t_0^5}\, \bm{p}_{t_0}\, \bm{p}_{t_0^3}\, \bm{p}_{t_0^4}\rVert\Delta(\bm{p})=A_2(\bm{p})A_{5,3}(\bm{p})-A_3(\bm{p})A_{5,2}(\bm{p}).
\end{align*}

This yields, by Lemma \ref{lem9},

\begin{equation}\label{eq10}
\lVert\bm{p}_{t_0^5}\, \bm{p}_{t_0}\, \bm{p}_{t_0^3}\, \bm{p}_{t_0^4}\rVert=I_2(\bm{p})A_{5,3}(\bm{p})-I_3(\bm{p})A_{5,2}(\bm{p}).
\end{equation}

Again applying Lemma \ref{lem13} to the vectors $\bm{p}_{t_0},\bm{p}_{t_0^5},\bm{p}_{t_0^2},\bm{p}_{t_0^3},\bm{p}_{t_0^4}$ and $\bm{p}_{t_1},\bm{p}_{t_0^2},\bm{p}_{t_0^3}$ and eliminating the zero determinants, we have
\begin{align*}
\lVert\bm{p}_{t_0^5}\, \bm{p}_{t_0^2}\, \bm{p}_{t_0^3}\, \bm{p}_{t_0^4}\rVert\lVert\bm{p}_{t_0}\, \bm{p}_{t_1}\, \bm{p}_{t_0^2}\, \bm{p}_{t_0^3}\rVert-\lVert\bm{p}_{t_0}\, \bm{p}_{t_0^2}\, \bm{p}_{t_0^3}\, \bm{p}_{t_0^4}\rVert\lVert\bm{p}_{t_0^5}\, \bm{p}_{t_1}\, \bm{p}_{t_0^2}\, \bm{p}_{t_0^3}\rVert-\lVert\bm{p}_{t_0^5}\, \bm{p}_{t_0^2}\, \bm{p}_{t_0^3}\, \bm{p}_{t_0}\rVert\lVert\bm{p}_{t_0^4}\, \bm{p}_{t_1}\, \bm{p}_{t_0^2}\, \bm{p}_{t_0^3}\rVert=0.
\end{align*}

By the definitions of $A_1(\bm{p}),A_2(\bm{p}),A_{5,1}(\bm{p}),A_{5,2}(\bm{p}),\Delta(\bm{p})$, we obtain

\begin{align*}
\lVert\bm{p}_{t_0^5}\, \bm{p}_{t_0^2}\, \bm{p}_{t_0^3}\, \bm{p}_{t_0^4}\rVert\Delta(\bm{p})=A_2(\bm{p})A_{5,1}(\bm{p})-A_1(\bm{p})A_{5,2}(\bm{p}).
\end{align*}

This yields, by Lemma \ref{lem9}, 

\begin{equation}\label{eq11}
\lVert\bm{p}_{t_0^5}\, \bm{p}_{t_0^2}\, \bm{p}_{t_0^3}\, \bm{p}_{t_0^4}\rVert=I_2(\bm{p})A_{5,1}(\bm{p})-I_1(\bm{p})A_{5,2}(\bm{p}).
\end{equation}

Combining \eqref{eq10} and \eqref{eq11}, we conclude that

\begin{equation*}
A_1(\bm{p}_{t_0}) =\dfrac{n-1}{t_1}(I_3(\bm{p})A_{5,2}(\bm{p})-I_2(\bm{p})A_{5,3}(\bm{p}))-\dfrac{t_0}{t_1}(I_2(\bm{p})A_{5,1}(\bm{p})-I_1(\bm{p})A_{5,2}(\bm{p})).
\end{equation*}
\end{proof}

{This last lemma is also required for proving the algebraic independence of the $I_i$.}

\begin{lemma}\label{lem15}
Let $\bm{C}$ be a rational algebraic curve properly parametrized by $\bm{p}$ satisfying \textit{hypotheses} \textit{(i-iv)}, and $\bm{c}_{j}, j=0,...,n$ be the coefficient vectors of $\bm{p}$. {Then at least one of the polynomials $A_1(\bm{p})$, $A_2(\bm{p})$, $A_3(\bm{p})$, $A_4(\bm{p})$, $\Delta(\bm{p})$ depend on $t_0$.}
\end{lemma}

\begin{proof}
We proceed by contradiction to prove that if neither of $A_1(\bm{p})$, $A_2(\bm{p})$, $A_3(\bm{p})$, $A_4(\bm{p})$, $\Delta(\bm{p})$ depend on $t_0$, then $\bm{c}_0=0$; since this cannot happen by Remark \ref{isnotzero}, the statement follows.

So let us assume that the $A_i(\bm{p})$ and $\Delta(\bm{p})$ do not depend on $t_0$. In order to show that under this assumption $\bm{c}_0=0$, we will use induction on $n$. Recall that hypothesis (iii) assumes that $n\geq 4$. Now for $n=4$, we have
\begin{align*}
\bm{p}(t_0,t_1)=\left(\sum_{i=0}^{4}{c_{i,0}t_0^{4-i}t_1^{i}},\sum_{i=0}^{4}{c_{i,1}t_0^{4-i}t_1^{i}},\sum_{i=0}^{4}{c_{i,2}t_0^{4-i}t_1^{i}},\sum_{i=0}^{4}{c_{i,3}t_0^{4-i}t_1^{i}}\right),
\end{align*}
where the $c_{j,i}$, $0 \leq i \leq4$ are the components of the vectors $\bm{c}_j$, $0 \leq j \leq3$. Now we can compute the determinant $\Delta(\bm{p})$ as
\begin{align*}
\Delta(\bm{p})=& -192\lVert\bm{c}_2\, \bm{c}_3\, \bm{c}_1\, \bm{c}_0\rVert t_0^4t_1^5-192\lVert\bm{c}_2\, \bm{c}_4\, \bm{c}_1\, \bm{c}_0\rVert t_0^3t_1^6+288\lVert\bm{c}_3\, \bm{c}_4\, \bm{c}_1\, \bm{c}_0\rVert t_0^2t_1^7\\
& +384\lVert\bm{c}_3\, \bm{c}_4\, \bm{c}_2\, \bm{c}_0\rVert t_0t_1^8+96\lVert\bm{c}_3\, \bm{c}_4\, \bm{c}_2\, \bm{c}_1\rVert t_1^9.
\end{align*}
Since by assumption $\Delta(\bm{p})$ does not depend on $t_0$, we get
\begin{align*}
\lVert\bm{c}_3\, \bm{c}_4\, \bm{c}_2\, \bm{c}_1\rVert\neq0, \lVert\bm{c}_3\, \bm{c}_4\, \bm{c}_2\, \bm{c}_0\rVert=0, \lVert\bm{c}_3\, \bm{c}_4\, \bm{c}_1\, \bm{c}_0\rVert=0, \lVert\bm{c}_2\, \bm{c}_4\, \bm{c}_1\, \bm{c}_0\rVert=0, \lVert\bm{c}_2\, \bm{c}_3\, \bm{c}_1\, \bm{c}_0\rVert=0.
\end{align*}
Since $\lVert\bm{c}_3\, \bm{c}_4\, \bm{c}_2\, \bm{c}_0\rVert=0$ and $\lVert\bm{c}_3\, \bm{c}_4\, \bm{c}_2\, \bm{c}_1\rVert\neq0$, there are scalars $\lambda_1,\lambda_2,\lambda_3$ such that $\bm{c}_0=\lambda_1\bm{c}_3+\lambda_2\bm{c}_4+\lambda_3\bm{c}_2$. Substituting $\bm{c}_0=\lambda_1\bm{c}_3+\lambda_2\bm{c}_4+\lambda_3\bm{c}_2$ in the vanishing determinants we conclude that $\lambda_1=\lambda_2=\lambda_3=0$, i.e., $\bm{c}_0=0$.

Assume that the lemma holds for $n$. Let us show that for a parametrization $\bm{p}$ with degree $n+1$ and coefficient vectors $\bm{c}_j$, $j=0,...,n+1$, we also have $\bm{c}_0=0$. Consider the parametrization $\bm{q}=\bm{p}_{t_0}$ with degree $n$ and coefficient vectors $\bm{c}'_j$, $j=0,...,n$.  
Since by assumption $A_i(\bm{p})$, $i\in\{1,2,3\}$ and $\Delta(\bm{p})$ do not depend on $t_0$, we have
\begin{equation}\label{eq12}
A_1(\bm{p})=k_1t_1^{4n-6},A_2(\bm{p})=k_2t_1^{4n-6},A_3(\bm{p})=k_3t_1^{4n-5}, \Delta(\bm{p})=k_0t_1^{4n-3},
\end{equation}
where $k_0,k_1,k_2,k_3$ are constants. The equation \eqref{eq10} and Lemma \ref{lem12} yield
\begin{equation}\label{eq13}
A_{5,1}(\bm{p})=nk_2t_1^{4n-7}, A_{5,2}(\bm{p})=0, A_{5,3}(\bm{p})=k_1t_1^{4n-6}-k_2t_0t_1^{4n-7}, A_{5,4}(\bm{p})=k_3t_1^{4n-5}.
\end{equation}

By using Lemma \ref{lem14} and the equations \eqref{eq12} and \eqref{eq13}, we obtain that $A_i(\bm{q})$, $i\in\{1,2,3\}$ and $\Delta(\bm{q})$ do not depend on $t_0$, so $\bm{c}'_0=0$. However, one can easily see that $\bm{c}_j=(n+1-j)\bm{c}'_j$ for all $j\in\{0,...,n\}$, and therefore $\bm{c}_0=0$.  
\end{proof}

Now we are ready to prove that the {$I_i$} are algebraically independent.

\begin{theorem}\label{teo17}
Let $\bm{C}$ be a rational algebraic curve properly parametrized by $\bm{p}=(p_1,p_2,p_3,p_4)$ satisfying \textit{hypotheses} \textit{(i-iv)}. Then $I_1(\bm{p}),I_2(\bm{p}),I_3(\bm{p}),I_4(\bm{p})$ are algebraically independent.
\end{theorem}

\begin{proof}
Let us assume that $I_1(\bm{p})=\dfrac{A_1(\bm{p})}{\Delta(\bm{p})},I_2(\bm{p})=\dfrac{A_2(\bm{p})}{\Delta(\bm{p})},I_3(\bm{p})=\dfrac{A_3(\bm{p})}{\Delta(\bm{p})},I_4(\bm{p})=\dfrac{A_4(\bm{p})}{\Delta(\bm{p})}$ are algebraically dependent. It follows that the homogeneous polynomials $\Delta(\bm{p}),A_1(\bm{p}),A_2(\bm{p}),A_3(\bm{p}),A_4(\bm{p})$ in $t_0,t_1$ are algebraically dependent. Thus, the Jacobian matrix
\begin{equation*}
J(\bm{p})=\begin{bmatrix}
\dfrac{\partial \Delta(\bm{p})}{\partial t_0} & \dfrac{\partial \Delta(\bm{p})}{\partial t_1} \\
\dfrac{\partial A_1(\bm{p})}{\partial t_0} & \dfrac{\partial A_1(\bm{p})}{\partial t_1} \\
\dfrac{\partial A_2(\bm{p})}{\partial t_0} & \dfrac{\partial A_2(\bm{p})}{\partial t_1} \\
\dfrac{\partial A_3(\bm{p})}{\partial t_0} & \dfrac{\partial A_3(\bm{p})}{\partial t_1} \\
\dfrac{\partial A_4(\bm{p})}{\partial t_0} & \dfrac{\partial A_4(\bm{p})}{\partial t_1}
\end{bmatrix}
\end{equation*}
has rank $1$. Note that the total degrees of the homogeneous polynomials $\Delta(\bm{p}),A_1(\bm{p}),A_2(\bm{p}),A_3(\bm{p}),A_4(\bm{p})$ are $4n-7,4n-10,4n-10,4n-9,4n-8$, respectively, where $n\geq 4$ is the degree of $\bm{p}$. Using Euler's Homogeneous Function Theorem to eliminate the partial derivative with respect to $t_1$ in the second column of $J(\bm{p})$, we can write $J(\bm{p})$ as
\begin{equation*}
J(\bm{p})=\begin{bmatrix}
\dfrac{\partial \Delta(\bm{p})}{\partial t_0} & \dfrac{4n-7}{t_1}\Delta(\bm{p})-\dfrac{t_0}{t_1}\dfrac{\partial \Delta(\bm{p})}{\partial t_0} \\
\dfrac{\partial A_1(\bm{p})}{\partial t_0} & \dfrac{4n-10}{t_1}A_1(\bm{p})-\dfrac{t_0}{t_1}\dfrac{\partial A_1(\bm{p})}{\partial t_0} \\
\dfrac{\partial A_2(\bm{p})}{\partial t_0} & \dfrac{4n-10}{t_1}A_2(\bm{p})-\dfrac{t_0}{t_1}\dfrac{\partial A_2(\bm{p})}{\partial t_0} \\
\dfrac{\partial A_3(\bm{p})}{\partial t_0} & \dfrac{4n-9}{t_1}A_3(\bm{p})-\dfrac{t_0}{t_1}\dfrac{\partial A_3(\bm{p})}{\partial t_0} \\
\dfrac{\partial A_4(\bm{p})}{\partial t_0} & \dfrac{4n-8}{t_1}A_4(\bm{p})-\dfrac{t_0}{t_1}\dfrac{\partial A_4(\bm{p})}{\partial t_0}
\end{bmatrix}.
\end{equation*}
Applying elementary operations by columns, we reach the matrix
\begin{equation*}
\tilde{J}(\bm{p})=\begin{bmatrix}
\dfrac{\partial \Delta(\bm{p})}{\partial t_0} & \dfrac{4n-7}{t_1}\Delta(\bm{p})\\
\dfrac{\partial A_1(\bm{p})}{\partial t_0} & \dfrac{4n-10}{t_1}A_1(\bm{p}) \\
\dfrac{\partial A_2(\bm{p})}{\partial t_0} & \dfrac{4n-10}{t_1}A_2(\bm{p}) \\
\dfrac{\partial A_3(\bm{p})}{\partial t_0} & \dfrac{4n-9}{t_1}A_3(\bm{p}) \\
\dfrac{\partial A_4(\bm{p})}{\partial t_0} & \dfrac{4n-8}{t_1}A_4(\bm{p})
\end{bmatrix},
\end{equation*}
which must also have rank $1$. Because of this, $\dfrac{4n-8}{t_1}\dfrac{\partial \Delta(\bm{p})}{\partial t_0}A_4(\bm{p})-\dfrac{4n-7}{t_1}\Delta(\bm{p})\dfrac{\partial A_4(\bm{p})}{\partial t_0}=0$. Solving this differential equation yields $\Delta(\bm{p})^{4n-8}=h(t_1)A_4(\bm{p})^{4n-7}$, where $h$ is an arbitrary function of $t_1$. But since the degrees of the homogeneous polynomials $\Delta(\bm{p})^{4n-8}$ and $A_4(\bm{p})^{4n-7}$ are the same, $h$ must be a constant function, say $h(t_1)=c$. By definition, $\dfrac{\partial \Delta(\bm{p})}{\partial t_0}=A_4(\bm{p})$. Therefore we have a differential equation $\Delta(\bm{p})^{4n-8}=c(\dfrac{\partial \Delta(\bm{p})}{\partial t_0})^{4n-7}$. Solving this equation, we get $\Delta(\bm{p})(t_0,t_1)=(c_1t_0+g(t_1))^{4n-7}$, where $g$ is an arbitrary function of $t_1$ and $c_1$ is a constant. We know that $\Delta(\bm{p})$ is a homogeneous polynomial in $t_0,t_1$ with a total degree $4n-7$, so $g$ must be of the form $g(t_1)=c_2t_1$. On the other hand, according to Lemma \ref{lem11}, $t_1$ must be a factor of $\Delta(\bm{p})$. Thus we have $c_1=0$, and in this case, there is a constant $r_0$ such that $\Delta(\bm{p})=r_0t_1^{4n-7}$. 

Again, since $\tilde{J}(\bm{p})$ has rank $1$, the following equations also hold
\begin{equation}\label{eq14}
\dfrac{4n-10}{t_1}\dfrac{\partial \Delta(\bm{p})}{\partial t_0}A_1(\bm{p})-\dfrac{4n-7}{t_1}\Delta(\bm{p})\dfrac{\partial A_1(\bm{p})}{\partial t_0}=0
\end{equation}
\begin{equation}\label{eq15}
\dfrac{4n-10}{t_1}\dfrac{\partial \Delta(\bm{p})}{\partial t_0}A_2(\bm{p})-\dfrac{4n-7}{t_1}\Delta(\bm{p})\dfrac{\partial A_2(p)}{\partial t_0}=0
\end{equation}
\begin{equation}\label{eq16}
\dfrac{4n-9}{t_1}\dfrac{\partial \Delta(\bm{p})}{\partial t_0}A_3(\bm{p})-\dfrac{4n-7}{t_1}\Delta(\bm{p})\dfrac{\partial A_3(\bm{p})}{\partial t_0}=0
\end{equation}
Using the fact that $\Delta(\bm{p})$ does not depend on $t_0$ and the equations \eqref{eq14}, \eqref{eq15}, \eqref{eq16}, we deduce that $A_i(\bm{p})$ for $i\in\{1,2,3\}$ do not depend on $t_0$. But this contradicts Lemma \ref{lem15}, and therefore the $I_i$ are algebraically independent.
\end{proof}

And we can finally prove Lemma \ref{denomnot}. 

\begin{proof} (of Lemma \ref{denomnot}) Assume that $I_3(\bm{p})$ and $I_4(\bm{p})$ are not identically zero. Then because of Theorem \ref{teo17}, the expression $8(n-3)I_3(\bm{p})+3(n-2)I_4^2(\bm{p})$ in the denominator of $\kappa_1,\kappa_2$ cannot be identically zero. So let us assume that both $I_3(\bm{p})$ and $I_4(\bm{p})$ are identically zero. Then, by the definition of the {$I_i$,} $A_3(\bm{p})$ and $A_4(\bm{p})$ are both identically zero. Since $A_3(\bm{p})=\Vert \bm{p}_{t_0}\, \bm{p}_{t_1}\, \bm{p}_{t_0^4}\, \bm{p}_{t_0^3} \Vert \equiv 0$ and $A_4(\bm{p})=\Vert \bm{p}_{t_0}\, \bm{p}_{t_1}\, \bm{p}_{t_0^2}\, \bm{p}_{t_0^4} \Vert \equiv 0$, we conclude that both of $\bm{p}_{t_0^2}$ and $\bm{p}_{t_0^3}$ are linear combinations of $\bm{p}_{t_0},\bm{p}_{t_0^2},\bm{p}_{t_0^4}$. Substituting these linear combinations in $A_1(\bm{p})$ and $A_2(\bm{p})$ yields that both $A_1(\bm{p})$ and $A_2(\bm{p})$ are identically zero too. But this again contradicts Theorem \ref{teo17}.
\end{proof}   
\end{document}